\documentclass{article}
\usepackage{amsmath}[1996/11/01]
\usepackage{amssymb,amsthm,amsxtra}

\def\[#1\]{\begin{equation}#1\end{equation}}
\makeatletter
\def\beq{%
   \relax\ifmmode
      \@badmath
   \else
      \ifvmode
         \nointerlineskip
         \makebox[.6\linewidth]%
      \fi
      $$
   \fi
}
\def\eeq{%
   \relax\ifmmode
      \ifinner
         \@badmath
      \else
         $$
      \fi
   \else
      \@badmath
   \fi
   \ignorespaces
}

\def\enddisplaymath{\eeq\global\@ignoretrue}
\makeatother

\newtheorem{thm}{Theorem}
\newtheorem{cor}[thm]{Corollary}
\newtheorem{lem}[thm]{Lemma}
\newtheorem{prop}[thm]{Proposition}
\newtheorem{conj}{Conjecture}
\theoremstyle{remark}
\newtheorem*{rem}{Remark}
\newtheorem{rems}{Remark}[thm]

\newtheorem{eg}{Example}
\theoremstyle{definition}
\newtheorem{defn}{Definition}

\numberwithin{equation}{section}
\numberwithin{thm}{section}
\numberwithin{eg}{section}

\DeclareMathOperator\hgt{ht}
\DeclareMathOperator\Ind{Ind}
\DeclareMathOperator\PS{PS}
\DeclareMathOperator\PGL{PGL}
\DeclareMathOperator\AGL{AGL}
\DeclareMathOperator\GL{GL}
\DeclareMathOperator\Sp{Sp}
\DeclareMathOperator\Mat{Mat}
\DeclareMathOperator\Alt{Alt}
\DeclareMathOperator\Sym{Sym}
\def\even{\widetilde}
\DeclareMathOperator\Stab{Stab}
\newcommand{\Z}{{\mathbb Z}}
\newcommand{\Q}{{\mathbb Q}}
\newcommand{\C}{{\mathbb C}}
\newcommand{\F}{{\mathbb F}}
\newcommand{\disj}{\uplus}
\newcommand{\1}{{\mathbf 1}}
\newcommand{\evenX}{{\tilde X}}
\newcommand{\evenx}{{\tilde x}}
\newcommand{\eveny}{{\tilde y}}
\newcommand{\semidirect}{\rtimes}

\begin{document}

\title{Deformations of permutation representations of Coxeter groups}
\author{Eric M. Rains \thanks{Partially supported by NSF grant DMS-0833464.} \\
Monica J. Vazirani \thanks{Partially supported by NSA grant H982300910076,
and by Caltech and DMS-0833464 during her visits.}}

\date{June 30, 2010}
\maketitle

\begin{abstract}
The permutation representation afforded by a Coxeter group $W$ acting                      
on the cosets of a standard parabolic subgroup inherits many nice                          
properties from $W$ such as
a shellable Bruhat                        
order and a flat deformation over $\Z[q]$ to a representation of the                        
corresponding Hecke algebra.  In this paper we define a larger class of                     
``quasiparabolic" subgroups (more generally, quasiparabolic $W$-sets),
and show that they also inherit these properties.  Our motivating                           
example is the action of the symmetric group on fixed-point-free                            
involutions by conjugation.  

\end{abstract}

\tableofcontents

\section{Introduction}

The motivation for the machinery developed in this paper arose when
the first author was investigating certain conjectures involving Macdonald
polynomials generalizing classical identities of Schur functions related to
the representation theory of symmetric spaces \cite{bcpoly}.  These Schur
function identities are closely related to Littlewood identities and thus
via invariant theory \cite[\S 8]{algebraic} to the action of $S_{2n}$ by
conjugation on fixed-point-free involutions.  Given the close connection
between Macdonald polynomials and Hecke algebras, this suggested that to
prove those conjectures, one should first deform that permutation
representation.

There is a straightforward way of deforming any given representation
of a finite Coxeter group to its corresponding Hecke algebra.
Indeed, when the ground field is algebraically closed of characteristic $0$
and the parameter $q$ is specialized to something other than a  nontrivial root of unity,
the Hecke algebra is isomorphic to the group algebra, 
and thus every irreducible representation deforms. 
So to deform any representation, one needs simply express it
as a direct sum of irreducibles and deform each irreducible. 
However, the resulting deformation is quite complicated:
this process does not give rise to any particularly nice basis.  
In addition,  this construction can fail in the case of infinite Coxeter groups,
as even their finite-dimensional representations can fail to be completely reducible.
One would thus like a version of this construction that works for arbitrary
Coxeter groups and in particular yields a basis for which the structure
constants are polynomials in $q$ rather than having poles at roots of unity. 

There is one special case of a permutation representation,
apart of course from the regular representation, which has
 a particularly nice deformation.
If $W_I \subset W$ is a parabolic subgroup,\footnote{Unless specifically noted
otherwise, parabolic means standard parabolic
throughout the paper, whereas a conjugate of a standard parabolic
subgroup will be called conjugate parabolic.}
then there is a classical construction of a module
deforming $W/W_I$.  To be precise, if we define
$\ell^I(w)$ to be the length of the shortest element
of the coset $w W_I$, then we can define an action of
the corresponding Hecke algebra $H_W(q)$ on
$\bigoplus_{w \in W^I} \Z[q] w W_I$ (where $W^I$ denotes the set of
minimal length left coset representatives) as follows.
\[
T(s) w W_I = \begin{cases}
sw W_I & \ell^I(sw)>\ell^I(w)\\
q w W_I  & \ell^I(sw)=\ell^I(w)\\
(q-1) w W_I +q sw W_I & \ell^I(sw)<\ell^I(w).
\end{cases}
\]
This representation can also be obtained by inducing
the trivial representation from the corresponding parabolic
subalgebra.  For our other deformations of permutation representations,
no such subalgebra exists.  We do not consider the natural problem of deforming
other induced representations.  

If one conjugates a fixed-point-free involution by a simple
transposition, the length either stays the same or changes
by $2$. This suggests constructing a Hecke algebra representation
along the same lines as in the parabolic case, replacing the
length function of a coset by half the length of
the fixed-point-free involution $\iota$, which we refer to as its
height. 
If one does this, it turns out that the result indeed satisfies
the relations of the Hecke algebra.
Similarly, if one instead uses the negated height 
$-\frac{\ell(\iota)}{2}$ one again obtains a well-defined
module, albeit with a rather different looking action.
(In fact, the two modules are indeed isomorphic over 
$\Z[q,q^{-1}]$.)
In either case the module is generated by the unique minimal (maximal)
length 
element, and thus is determined by the ideal
annihilating that element. 
 
It then turned out  \cite{vanish} that one could use the most natural extensions
of these annihilator ideals to the affine Hecke algebra to settle the original
Macdonald polynomial conjectures of the first author. 
It turns out that the existence of a deformation of the above form is 
is actually a fairly stringent condition on a subgroup of a Coxeter 
group.
It remains conjectural that these extended ideals are annihilators
of minimal elements in modules of the affine Hecke algebra coming
from the present construction. 

The present work came about in an attempt to systematically
understand which properties of the set of fixed-point-free involutions
permit such a nice deformation.

One other property of the quotient by a parabolic subgroup is
particularly notable.
Namely, the existence of a very well-behaved partial order induced
by the Bruhat order on $W$.
It was thus natural to look for a corresponding partial order in
the case of fixed-point-free involutions as well as other
permutation representations which admit natural deformations.

With only two exceptions, all of the representations we had found
indeed admitted an analogue of Bruhat order.  
This led us to formulate Definition \ref{defn:qp} below.
Since parabolic subgroups were our prototypical example
of such subgroups, we called our larger class of subgroups
{\em quasiparabolic}.\footnote{After having developed much of the theory presented
  here, we discovered that there was an unrelated notion of
  ``quasi-parabolic'' subgroup for classical Weyl groups already in the
  literature (J. Du, L. Scott, Trans. AMS 352(9), 4325--4353), but by then
  we had grown far too attached to the name to change it!}
Note that just as conjugates of standard parabolic subgroups do not give
rise to nice Hecke algebra modules, we similarly find that conjugates
of quasiparabolic subgroups are rarely quasiparabolic.

Although our main examples of permutation representations take the
form $W/H$ for various subgroups $H$, we take a slightly more
general approach and formulate at least the definition for arbitrary
sets with Coxeter group actions.

In Section \ref{sec:defns} after giving the definitions, we show
that quasiparabolic $W$-sets satisfy an analogue of the 
Strong Exchange Condition
(in a slightly odd form because  there is no notion of reduced word
in a general quasiparabolic set)
and consider other elementary consequences of our definition.

Section \ref{sec:constrs} gives a number of constructions
of quasiparabolic sets, both from Coxeter groups themselves
and from other quasiparabolic sets.   
In particular, we show that any parabolic subgroup is
quasiparabolic (as a very special case of Corollary \ref{cor:induce_from_para}
below), as is the regular representation of $W \times W$ on $W$.
More generally there are combinatorial analogues  of the module-theoretic operations of
 restriction to parabolics, induction from parabolics, and tensor product. 

One of these constructions is:   given a quasiparabolic set with odd
length stabilizers, there is a natural quasiparabolic double cover
with only even length stabilizers. 
This enables us to reduce many of our arguments to the even case.
In particular one finds that the even subgroups of quasiparabolic
subgroups share many of the same properties.

Our final construction (in Section \ref{sec:inv}) 
generalizes the case of fixed-point-free involutions
by giving fairly general conditions under which a conjugacy class
of involutions becomes quasiparabolic with the obvious height function.
(Unfortunately the resulting notion does not appear adequate
to address involutions in infinite Coxeter groups.)

Section \ref{sec:bruhat}
introduces the analogue of Bruhat order and again deals with basic
properties.  The main result of that section  is Theorem \ref{thm:bruhat_para}
which shows that the Bruhat order is essentially unchanged under restriction
to a parabolic subgroup.  We also characterize how the double cover
construction affects Bruhat order.
Apart from these two results (and the straightforward fact that
our Bruhat order on the actions of $W \times W $ on $W$ and  of $W$ on
$W/W_I$ agree with the usual Bruhat order),
we have not attempted to understand how our constructions affect Bruhat
order in general. 

Even without a full understanding of Bruhat order,
we are still able to show that two of the more important topological
properties of Bruhat order apply to a general quasiparabolic subgroup.
The two main results of Section \ref{sec:top} are
Theorem \ref{thm:moeb_even} (and its Corollary \ref{cor:moeb_odd})
which give a formula for the M\"obius function of the Bruhat order,
generalizing results of Verma and Deodhar \cite{VermaD-N:1971,DeodharVV:1977}.
Note that our proof is somewhat simplified by the ability to reduce
to the even stabilizer case when the poset is Eulerian.
Our other topological result generalizes a theorem of Bj\"orner and Wachs
by showing that for any interval in a (bounded) quasiparabolic
set, the corresponding order complex is shellable and in fact  homeomorphic to a sphere or
to a cell (depending on the Euler characteristic). 

In Section \ref{sec:hecke} we show that the construction suggested above
indeed gives well-defined modules over the Hecke algebra. 
We also show that these modules admit a natural bilinear form
induced by the height function. 
We also give an algorithm for finding small sets of generators
of the annihilators of minimal elements of these modules
and in particular show that the ideal associated to a finite
quasiparabolic set is finitely generated even when the Coxeter
group that is acting has infinite order.

In Section \ref{sec:PS} we study Poincar\'e series, i.e., the generating
function of height, of quasiparabolic sets over finite Coxeter groups
and show these generating functions can be controlled by Hecke algebra modules.
We find both that such Poincar\'e series are always palindromes
(despite the fact that Bruhat order rarely admits an order-reversing
automorphism) and that the Poincar\'e series of a quasiparabolic
$W$-set always divides the Poincar\'e series of $W$, despite
the absence of any combinatorial interpretation of the quotient.
The palindrome property suggests the existence of order-reversing automorphisms of
the Hecke algebra modules; we give an explicit conjecture along these lines. 
Note that automorphisms of the form prescribed by the conjecture would give rise to
analogues of the $R$-polynomials
of Kazhdan-Lusztig theory. 

Finally, in Section \ref{sec:egs}, we discuss a number of examples of
finite quasiparabolic sets, including several quasiparabolic
subgroups of finite groups which do not come from any of our general
constructions. 
We have fairly extensively explored the set of quasiparabolic subgroups
of finite Coxeter groups (in rank $ \le 8$) but space does not permit
an exhaustive discussion. We thus focus primarily on the analogue of
fixed-point-free involutions and cases that exhibit suggestive phenomena.

There are a number of open problems scattered throughout the paper.
In the interest of timeliness we focused our efforts on developing
enough of a theory to show that the class of quasiparabolic subgroups
is a natural one.  There are also some implicit open problems
such as classifying all quasiparabolic subgroups of all finite
Coxeter groups.  Though significant progress could most likely be made
on this last problem without substantially new ideas, other 
problems such as the existence of $R$-polynomials appear to require
new insight.

For basic results and notation on Coxeter groups we refer the reader
to \cite{HumphreysJE:1990}. 
As there, we will only consider Coxeter groups of finite rank. 
We denote by $R(W)$ the set of reflections (i.e., conjugates of the simple
reflections $s\in S$) of a Coxeter system $(W,S)$.
We denote by $W^0 = \{w \in W \mid \ell(w) \equiv 0 \pmod 2\}$ the {\em even}
subgroup of $W$, and more generally for any $H \subset W$ let $H^0 = H \cap W^0$.
For instance, in the case of the symmetric group $S_n$ its even subgroup
is the alternating group, denoted here $\Alt_n$, to avoid confusion with $A_n$.
\section{Definitions}\label{sec:defns}

\begin{defn}
Let $(W,S)$ be a Coxeter system.  A {\em scaled $W$-set} is a pair
$(X,\hgt)$ with $X$ a $W$-set and $\hgt:X\to \Z$ a function such that
$|\hgt(sx)-\hgt(x)|\le 1$ for all $s\in S$.  An element $x\in X$ is
$W$-{\em minimal} if $\hgt(sx)\ge \hgt(x)$ for all $s\in S$, and similarly for
$W$-{\em maximal}.
\end{defn}

This is invariant under shifting and negation of heights; more precisely:

\begin{prop}
Let $(X,\hgt)$ be a scaled $W$-set.  Then for any $k\in \Z$, the new height
functions $(k+\hgt)(x):=k+\hgt(x)$ and $(k-\hgt)(x):=k-\hgt(x)$ make
$(X,k+\hgt)$ and $(X,k-\hgt)$ scaled $W$-sets.
\end{prop}

\begin{defn}\label{defn:qp}
A {\em quasiparabolic} $W$-set is a scaled $W$-set $X$ satisfying the
following two properties:
\begin{list}{\labelitemi}{\leftmargin=1em}
\item[(QP1)] For all $r\in R(W)$, $x\in X$, if $\hgt(rx)=\hgt(x)$, then $rx=x$.
\item[(QP2)] For all $r\in R(W)$, $x\in X$, $s\in S$, if $\hgt(rx)>\hgt(x)$ and
$\hgt(srx)<\hgt(sx)$, then $rx=sx$.
\end{list}
\end{defn}

The main motivating example of a quasiparabolic $W$-set is the homogeneous
space $W/W_I$ for a (standard) parabolic subgroup $W_I$, where the height of
a coset is the length of its minimal representative; see Corollary
\ref{cor:induce_from_para} below.  In general, as we will see, many of the
well-known properties of these parabolic homogeneous spaces extend to
general quasiparabolic $W$-sets.  (This explains our choice of
terminology.)

The case $I=\emptyset$ gives the {\em regular representation} of $W$; i.e.,
the action of $W$ on itself by left multiplication, with height given by
length.  This extends to the action of $W\times W$ on $W$ by left and right
multiplication, see Theorem \ref{thm:diag_is_qp} below.

The situation of property QP2 is fairly rigid, as seen in the following Lemma.

\begin{lem}\label{lem:useful_lemma}
Let $(X,\hgt)$ be a scaled $W$-set, and suppose $r\in R(W)$, $s\in S$,
$x\in X$ are such that $\hgt(rx)>\hgt(x)$ and $\hgt(srx)<\hgt(sx)$.  Then
$\hgt(rx)=\hgt(sx)=\hgt(x)+1=\hgt(srx)+1$.
\end{lem}

\begin{proof}
Since $(X,\hgt)$ is a scaled $W$-set, it follows that
$|\hgt(sx)-\hgt(x)|\le 1$ and $|\hgt(rx)-\hgt(srx)|\le 1$.  Since
$\hgt(rx)-\hgt(x)\ge 1$ (by integrality) and similarly
$\hgt(sx)-\hgt(srx)\ge 1$, the triangle inequality forces all differences
to be 1 as claimed.
\end{proof}

\begin{prop}
If $(X,\hgt)$ is quasiparabolic, then so are $(X,k+\hgt)$ and $(X,k-\hgt)$ for
any $k\in \Z$.
\end{prop}

\begin{proof}
Shifting the height clearly doesn't affect either property QP1 or QP2, so it
suffices to consider the negated height function $-\hgt$.  Property QP1 is
again preserved by this, but property QP2 appears to be asymmetrical.  However,
we observe that if $\hgt(rx)>\hgt(x)$ and $\hgt(srx)<\hgt(sx)$, then replacing
$r$ by $srs$ and $x$ by $sx$ reverses the inequalities, so property QP2 is in
fact symmetrical as well.
\end{proof}

\begin{prop}
If $(X,\hgt)$ is a quasiparabolic $W$-set, then for any 
parabolic
subgroup $W_I\subset W$, $(X,\hgt)$ is a quasiparabolic $W_I$-set.
\end{prop}

\begin{proof}
Follows immediately from the fact that $R(W_I)\subset R(W)$.
\end{proof}

Note that in the case of the regular representation of $W$, with $x=1$, the
following ``strong exchange condition'' becomes the usual strong exchange
condition.

\begin{thm}[Strong Exchange]\label{thm:strong_exch}
Let $(X,\hgt)$ be a quasiparabolic $W$-set and  let $x \in X$.
Let $w=s_1\cdots s_k\in W$ be an arbitrary word.  If
$\hgt(rx)>\hgt(x)$ and $\hgt(wrx)<\hgt(wx)$, then there exists $1\le i\le k$
such that $wrx = s_1\cdots s_{i-1}s_{i+1}\cdots s_k x$.
\end{thm}

\begin{proof}
Let $x_i = s_{i+1}\cdots s_k x$, $y_i = s_{i+1}\cdots s_k r x = r_i x_i$,
where $r_i$ is the reflection
\[
r_i = s_{i+1}\cdots s_k r s_k \cdots s_{i+1},
\]
and observe that $\hgt(y_0)<\hgt(x_0)$ and $\hgt(y_k)>\hgt(x_k)$, so there
exists $1\le i\le k$ such that $\hgt(y_{i-1})<\hgt(x_{i-1})$ but $\hgt(y_i)\ge
\hgt(x_i)$.  If equality holds, then $y_i=r_ix_i=x_i$, so
$y_{i-1}=s_{i-1}y_i=s_{i-1}x_i=x_{i-1}$, a contradiction.  But then we may
apply property QP2 to deduce that
\[
y_i = r_i x_i = s_i x_i.
\]
The theorem follows.
\end{proof}

\begin{thm}\label{thm:minimal_eq_minimum}
Let $x$ be a $W$-minimal element of the quasiparabolic $W$-set
$(X,\hgt)$.  Then for all $w\in W$, $\hgt(wx)\ge \hgt(x)$, with equality iff $wx=x$.
\end{thm}

\begin{proof}
Suppose otherwise, and let $w=s_1s_2\cdots s_k$ be a counterexample of
minimum length.  In particular, $\hgt(s_k x)=\hgt(x)+1$, since otherwise $s_k
x=x$, and $ws_k$ is a counterexample of length $k-1$.  Let $1\le j<k$ thus
be the largest index such that
\[
\hgt(s_js_{j+1}\cdots s_k x)\le \hgt(s_{j+1}\cdots s_k x).
\]
If equality holds here, then again we can remove $s_j$ from $w$ to obtain a
shorter counterexample, so
\[
\hgt(s_js_{j+1}\cdots s_k x)=\hgt(s_{j+1}\cdots s_k x)-1.
\]
Let $r$ be the reflection
\[
r = s_k \cdots s_{j+1} s_j s_{j+1}\cdots s_k.
\]
Then by assumption
\[
\hgt(s_{j+1}\cdots s_k r x) = \hgt(s_j s_{j+1}\cdots s_k x) = \hgt(s_{j+1}\cdots
s_k x)-1,
\]
while
\[
\hgt(rx) \ge \hgt(s_j s_{j+1}\cdots s_k x)-(k-j) = \hgt(x)+1>\hgt(x),
\]
where the second equality follows from maximality of $j$.  We may thus
apply strong exchange to conclude that
\[
s_j s_{j+1}\cdots s_k x
= s_{j+1}\cdots s_k r x = s_{j+1}\cdots s_{l-1} s_{l+1}\cdots s_k x
\]
for some $l$.  But then $s_j$ and $s_l$ can be removed from $w$ without
affecting $wx$, contradicting minimality of $w$.
\end{proof}

\begin{rem}
Of course, the analogous argument applies in the case of a $W$-maximal
element, by symmetry.
\end{rem}

\begin{cor}
Each orbit of a quasiparabolic $W$-set contains at most one $W$-minimal and
at most one $W$-maximal element.
\end{cor}

\begin{proof}
If $x$ and $wx$ are $W$-minimal elements, then $\hgt(wx)\ge \hgt(x)$ and
$\hgt(x)=\hgt(w^{-1}(wx))\ge \hgt(wx)$, so $\hgt(wx)=\hgt(x)$ and thus
$wx=x$.
\end{proof}

\begin{defn}
Let $(X,\hgt)$ be a quasiparabolic $W$-set with $W$-minimal element
$x_0$.  Suppose $x \in X$ has height $\hgt(x)  = k + \hgt(x_0)$.  Then 
we call $s_1 s_2 \cdots s_k x_0$ a {\em reduced expression} for $x$ if
$x = s_1 s_2 \cdots s_k  x_0$,
$s_i \in S$.
By abuse of notation, we also call $w x_0$  reduced where $w = s_1 s_2 \cdots s_k $
is the corresponding reduced word.
\end{defn}

In the case that $X$ has a $W$-minimal element, the strong exchange
condition can be restated in the following way, which is more traditional
in the parabolic case.

\begin{cor}\label{cor:strong_exch}
If $x_0$ is $W$-minimal in the quasiparabolic $W$-set $(X,\hgt)$, and $r\in R(W)$,
$w=s_1\cdots s_k\in W$ are such that $\hgt(wx_0)>\hgt(rwx_0)$, then there exists
$1\le i\le k$ such that
\[
rwx_0 = s_1\cdots s_{i-1}s_{i+1}\cdots s_k x_0.
\]
If $wx_0$ is a reduced expression 
then $i$ is unique.
\end{cor}

\begin{proof}
Apply strong exchange to the reflection $w^{-1}rw$; by minimality,
$\hgt(x_0)\le \hgt(w^{-1}rwx_0)$, and equality would imply $wx_0=rwx_0$.
If $i$ is not unique, so that
\[
s_1\cdots s_{i-1}s_{i+1}\cdots s_k x_0
=
s_1\cdots s_{j-1}s_{j+1}\cdots s_k x_0
\]
for some $j$, then
\[
s_1\cdots s_{i-1}s_{i+1}\cdots s_{j-1}s_{j+1}\cdots s_k x_0
=
wx_0,
\]
and $wx_0$ is not reduced.
\end{proof}

In the case that $X$ is transitive, or equivalently is a homogeneous space,
with a minimal element, the height function is essentially just given by
the length of a minimal coset representative (as in the parabolic case).
More precisely, we have the following.

\begin{cor}\label{cor:ht_from_stab}
Let $x_0$ be a $W$-minimal element of the transitive quasiparabolic $W$-set
$(X,\hgt)$.  Then for all $y\in X$,
\[
\hgt(y) = \hgt(x_0)+\min_{wx_0=y}\{\ell(w)\}.
\]
\end{cor}

\begin{proof}
We may assume $y\ne x_0$, and therefore $y$ cannot be $W$-minimal, so there
exists $s\in S$ such that $\hgt(sy)=\hgt(y)-1$.  The claim follows by
induction on the height of $y$.
\end{proof}

\begin{rem}
Note that the proof depended on quasiparabolicity only via the fact that
$x_0$ is the unique $W$-minimal element of $X$.
\end{rem}

In particular, a transitive quasiparabolic $W$-set with a minimal element
is uniquely determined (up to an overall shift in height) by the stabilizer
of that minimal element. More precisely, we have the following.

\begin{prop}\label{prop:qp_subgp_to_set}
Let $(X,\hgt)$ and $(X',\hgt')$ be transitive quasiparabolic $W$-sets with minimal
elements $x_0$, $x_0'$ respectively.  If $\Stab_W(x_0)=\Stab_W(x_0')$ and
$\hgt(x_0)=\hgt'(x_0')$, then $X$ and $X'$ are isomorphic scaled $W$-sets.
\end{prop}

\begin{proof}
We construct the isomorphism $\phi:X\to X'$ as follows.  Let $y\in X$, and
let $g\in W$ be such that $y=gx_0$; then we take $\phi(gx_0)=gx_0'$.  To see that
this is well-defined, observe that for fixed $y$, $g$ is determined up to
right multiplication by an element of the stabilizer of $x_0$, and this does
not change $gx_0'$, as $x_0'$ has the same stabilizer.  It thus remains only to
show that this $W$-set isomorphism (since the analogous map $X'\to X$ is
clearly inverse to $\phi$) preserves height, which follows by the
observation
\[
\hgt(y) = \hgt(x_0)+\min_{wx_0=y}\{\ell(w)\}
       = \hgt'(x_0') + \min_{wx_0'=\phi(y)}\{\ell(w)\}
       = \hgt'(\phi(y)).
\]
\end{proof}

Given a subgroup $H\subset W$, we can construct a scaled $W$-set $W/H$ by
defining
\[
\hgt(wH) = \min_{v\in wH}\{\ell(v)\}.
\]
Note that $W/H$ has a $W$-minimal element, namely the trivial coset $H$ of
height $0$.

\begin{defn}
A subgroup $H\subset W$ is {\em quasiparabolic} if the scaled $W$-set $W/H$ is
quasiparabolic.
\end{defn}

We remark that our motivating examples can be restated as saying that both
standard parabolic subgroups and the diagonal subgroup of $W\times W$ are
quasiparabolic.

It follows from the previous corollary that if $X$ is a transitive
quasiparabolic $W$-set with minimal element $x_0$ of height $0$, then $X\cong
W/\Stab_W(x_0)$, and thus $\Stab_W(x_0)$ is a quasiparabolic subgroup of $W$.
Thus the quasiparabolic subgroups are precisely the stabilizers of
$W$-minimal elements of quasiparabolic $W$-sets.  (Note that via negation,
every statement about $W$-minimal elements applies mutatis mutandum to
$W$-maximal elements; e.g., the stabilizer of a $W$-maximal element is also
quasiparabolic.)

It is important to note that the property of being a quasiparabolic
subgroup is not invariant under conjugation.  Indeed, the conjugate of a
standard parabolic subgroup need not be quasiparabolic. 
\begin{eg}  For instance, the
scaled $W$-set associated to the conjugate parabolic subgroup $\langle
s_1s_2s_1\rangle$ of $S_3$ cannot be quasiparabolic, since it has two
distinct $W$-maximal elements.
\end{eg}

If the height function of a quasiparabolic $W$-set is bounded from below,
then each orbit has an element of minimum height, which is necessarily
$W$-minimal, so unique in the orbit.  In particular, the transitive
quasiparabolic $W$-sets with height bounded from below (e.g., if
$|X|<\infty$) are (up to isomorphism and shift of height) in one-to-one
correspondence with the quasiparabolic subgroups of $W$.  Many of our
proofs require an assumption of boundedness (either from above or below),
making this case particularly important.  However, there are also
interesting examples of unbounded quasiparabolic $W$-sets, so we have
attempted to avoid as much as possible the assumption of existence of a
$W$-minimal element.  Note also that if $X$ (transitive and quasiparabolic)
is infinite, then it cannot have {\em both} a $W$-minimal and a $W$-maximal
element.  Indeed, if $X$ (without loss of generality) has a $W$-minimal
element (i.e., $X=W/H$ up to shifting), then its height function is
unbounded from above, as there are only finitely many potential minimal
coset representatives of any fixed length.

\begin{prop}
Suppose the transitive scaled $W$-set $(X,\hgt)$ has a unique $W$-minimal element,
and the stabilizer of that element is a quasiparabolic subgroup.  Then $X$
is quasiparabolic.
\end{prop}

\begin{proof}
Without loss of generality, we may assume that the minimal element $x_0\in X$
has height 0, and we need simply show that $X\cong W/H$ as scaled $W$-sets,
where $H$ is the stabilizer of the minimal element.  In other words, 
we need to show that for $y\in X$,
\[
\hgt(y)=\min_{wx_0=y}\{\ell(w)\};
\]
this follows from the proof of Corollary \ref{cor:ht_from_stab}.
\end{proof}

\begin{lem}\label{lem:odd_qp_has_simple}
Suppose the quasiparabolic subgroup $H\subset W$ contains an element of odd
length.  Then it contains a simple reflection.
\end{lem}

\begin{proof}
Let $w=s_1\cdots s_k\in H$ have odd length, and consider the sequence
\[
\hgt(H),\ \hgt(s_k H),\ \hgt(s_{k-1}s_k H),\dots,\ \hgt(w H)=\hgt(H).
\]
Since $k$ is odd, we find that there exists some $j$ such that
\[
\hgt(s_j\cdots s_k H) = \hgt(s_{j+1}\cdots s_k H),
\]
and thus
\[
s_j\cdots s_k H = s_{j+1}\cdots s_k H,
\]
and so $W/H$ contains an element fixed by a simple reflection.  Let
$gH\in W/H$ be an element of minimum height among all those fixed by a
simple reflection $s$.  If $g=1$, we are done; otherwise, there exists
$s'\in S$ such that $\hgt(s' g H)<\hgt(g H)$.  Consider the orbit of $g H$
under the action of the standard parabolic subgroup $\langle
s,s'\rangle$.  The heights in this orbit are bounded both above and below,
so the orbit is finite, and the action of $s s'$ on the orbit has finite
order, say $k$.  Now, consider the sequence
\[
g H,\ s' g H,\ s s' g H,\ s' s s' g H,\dots,\ s'(s s')^{k-1} g H,\ (s s')^k
g H.
\]
By assumption, the last term in the sequence is $g H$, so the penultimate
term is $s g H = g H$.  Since the sequence has even length and the initial
and final terms have the same height, it follows that there must be an even
number of steps where 
the height stays the same, so at least one such step
which is not the last.  Each such step provides an element fixed by either
$s$ or $s'$; if fixed by $s'$, the element is certainly not $g H$, while if
fixed by $s$, it has the form $(s s')^l g H$ for $l<k$ so by minimality of
$k$ is again not $g H$.  In other words, the $\langle s,s'\rangle$-orbit
contains an element different from $g H$ which is also fixed by a simple
reflection; since $g H$ is $\langle s,s'\rangle$-maximal, this other
element has strictly smaller height, a contradiction.
\end{proof}
\section{Constructions}\label{sec:constrs}

Two trivial, but useful, examples of quasiparabolic subgroups of $W$ are
$W$ itself and its even subgroup $W^0$.  Less trivially, of our two
motivating examples above, the parabolic case $W_I$ has a natural
generalization, so to avoid duplication of effort, we prove it as a
corollary of that generalization.  For the second example, we have the
following proof.

\begin{thm}\label{thm:diag_is_qp}
Let $(W,S)$ be a Coxeter system.  The set $W$ together with the height
function $\hgt(w):=\ell(w)$ and the $W\times W$-action
\[
(w',w'')w = w' w (w'')^{-1},
\]
is a quasiparabolic $W\times W$-set.
\end{thm}

\begin{proof}
Since $|\ell(sw)-\ell(w)|=|\ell(ws)-\ell(w)|=1$, $(W,\ell)$ is indeed a
scaled $W\times W$-set, and furthermore
\[
\hgt((w',w'')w)-\hgt(w)\equiv \ell((w',w''))\pmod{2}.
\]
In particular, property QP1 is vacuous.  For
property QP2, there would normally be four cases (depending on which
factor of $W\times W$ $s$ and $r$ belong to), but these reduce to two by
symmetry, and then to one via the observation that
\[
r w = w (w^{-1} r w).
\]
We thus reduce to showing that if $r\in R(W)$, $s\in S$, $w\in W$ are such
that
\[
\ell(wr)>\ell(w)\quad\text{and}\quad \ell(swr)<\ell(sw),
\]
then $sw=wr$.  Since $\ell(swr)<\ell(sw)$, the classical strong exchange
condition tells us that, taking $w=s_1s_2\cdots s_k$, either
$swr = w$ (and we are done), or
\[
swr =  s_1\cdots s_{l-1}s_{l+1}\cdots s_k.
\]
But then
\[
wr = s s_1\cdots s_{l-1}s_{l+1}\cdots s_k,
\]
so $\ell(wr)=k-1$, a contradiction.
\end{proof}

\begin{cor}
The diagonal subgroup $\Delta(W)\subset W\times W$ is quasiparabolic.
\end{cor}

\begin{proof}
Indeed, this is the stabilizer of the minimal element $1\in W$.
\end{proof}

One important, if mostly trivial, construction for quasiparabolic $W$-sets
is direct product.

\begin{prop}
Let $(W,S)$ and $(W',S')$ be Coxeter systems, let $(X,\hgt)$ be a
quasiparabolic $W$-set, and let $(X',\hgt')$ be a quasiparabolic $W'$-set.
The height function $(\hgt\times\hgt')(x,x'):=\hgt(x)+\hgt'(x')$ on $X\times
X'$ makes $X\times X'$ a quasiparabolic $W\times W'$-set.
\end{prop}

\begin{proof}
We observe that any reflection in $W\times W'$ is either a reflection in
$W$ or a reflection in $W'$, and the restriction of $X\times X'$ to $W$ is
the disjoint union of $|X'|$ copies of $X$ (with appropriately shifted
heights).  Thus property QP1 is automatic, while property QP2 reduces
immediately to the case $r\in R(W)$, $s\in S'$.  But we find
\begin{align}
(\hgt\times\hgt')((r,1)(1,s)(x,x'))
-
(\hgt\times\hgt')((1,s)(x,x'))
&=
\hgt(rx)-\hgt(x)\notag\\
&=
(\hgt\times\hgt')((r,1)(x,x'))
-
(\hgt\times\hgt')((x,x')),
\end{align}
and thus property QP2 is vacuous in this case.
\end{proof}

If the scaled $W$-set $(X,\hgt)$ satisfies $\hgt(sx)\ne \hgt(x)$ for all
$s\in S$, $x\in X$, then the verification that $X$ is quasiparabolic is
noticeably simplified, as property QP1 is vacuous.  Indeed, under this
assumption,
\[
\hgt(wx)\equiv \hgt(x)+\ell(w) \pmod{2},\label{eq:even}
\]
for all $w\in W$, $x\in X$.  This motivates the following definition.

\begin{defn}
The scaled $W$-set $(X,\hgt)$ is {\em even} if for any pair $w\in W$,
$x\in X$ s.t. $wx=x$, one has $\ell(w)$ even.  Otherwise, we say that
$(X,\hgt)$ is {\em odd}.
\end{defn}

It turns out that associated to any scaled $W$-set $X$ is a canonically
defined even scaled $W\times A_1$-set $\evenX$ (the {\em even double cover}
of $X$) such that $X$ is quasiparabolic iff $\evenX$ is quasiparabolic;
thus in many proofs below, we can reduce our consideration to the even
case.

The scaled $W\times A_1$-set $\evenX$ is defined as follows.  As a set
$\evenX=X\times \F_2$, with $W$-action
\[
w(x,k) = (wx,k+\ell(w))
\]
and $A_1  = \langle s_0 \rangle$-action
\[
s_0(x,k)=(x,k+1).
\]
The height function on $\evenX$ is defined by
\[
\even\hgt(x,k) = \begin{cases}
\hgt(x) & \hgt(x)\equiv k\pmod{2}\\
\hgt(x)+1 & \hgt(x)\equiv k+1\pmod{2}.
\end{cases}
\]
We also let $\even x$ denote the preimage $(x,\hgt(x))\in \evenX$.
Note that $\even x$ is $A_1$-minimal whereas $s_0 \even x$ is $A_1$-maximal.

\begin{thm}
The pair $(\evenX,\even\hgt)$ is an even scaled $W\times A_1$-set, and is
quasiparabolic iff $(X,\hgt)$ is quasiparabolic.
\end{thm}

\begin{proof}
To see that $(\evenX,\even\hgt)$ is even and scaled, we need simply observe
that if $\hgt(sx)=\hgt(x)\pm 1$, then $\even\hgt(s(x,k))=\even\hgt(x,k)\pm 1$,
while if $\hgt(sx)=\hgt(x)$, then $\even\hgt(s(x,k))-\even\hgt(x,k)$ is $1$ or
$-1$ depending on whether or not $\hgt(x)\equiv k$ modulo 2, and similarly
for $s_0(x,k)$.

Now, suppose that $(X,\hgt)$ is quasiparabolic.  We must show that for all
$r\in R(W\times A_1)$, $s\in S\disj \{s_0\}$, $(x,k)\in \evenX$, if
$\even\hgt(r(x,k))>\even\hgt((x,k))$ and $\even\hgt(sr(x,k))<\even\hgt(s(x,k))$,
then $r(x,k)=s(x,k)$.  By parity considerations, $r(x,k)=(rx,k+1)$ and
$s(x,k)=(sx,k+1)$, so it remains only to show that $rx=sx$.  There are four
cases, depending on whether $r$ or $s$ are in $W$ or $A_1$.  If both are in
$W$, then we certainly have
\[
\hgt(rx)\ge \hgt(x)\quad\text{and}\quad\hgt(srx)\le \hgt(sx),
\]
since $|\even\hgt((x,k))-\hgt(x)|\le 1$.  But then by the quasiparabolicity of
$X$, we either have $rx=x$ or $rx=sx$.  In the latter case, we are done; in
the former, we have
\[
\even\hgt((x,k+1))>\even\hgt((x,k))\quad\text{and}\quad
\even\hgt((sx,k))<\even\hgt((sx,k+1)),
\]
so that $k \equiv \hgt(x)\pmod 2$ and $k\equiv \hgt(sx)\pmod 2$,
implying $\hgt(sx)=\hgt(x)$ and $sx=x=rx$ as required.
 Similarly, if both are in $A_1$, the verification is trivial. 
If $r=s_0$ (the unique reflection in $A_1$) and
$s\in S$, then $r$ commutes with $S$ and is a simple reflection, so this is
a special case of $r\in R(W)$, $s=s_0$.

In that remaining case, we have
\[
\even\hgt((rx,k+1))>\even\hgt((x,k))\quad\text{and}\quad
\even\hgt((rx,k))<\even\hgt((x,k+1)).
\]
The left- and right-hand sides of the two inequalities both differ by at
most 1, so the only way to have opposite inequalities is to have
\[
\even\hgt((x,k+1))=\even\hgt((rx,k+1))
=\even\hgt((x,k))+1=\even\hgt((rx,k))+1,
\]
implying $\hgt(rx)=\hgt(x)$, so $rx=x$ and $r(x,k)=(x,k+1)=s_0(x,k)$ as
required.

Conversely, suppose $(\evenX,\even\hgt)$ is quasiparabolic; we need to show
that $(X,\hgt)$ satisfies properties QP1 and QP2.

For QP1, suppose $r\in R(W)$, $x\in X$ such that $\hgt(rx)=\hgt(x)$, and
observe that
\[
\even\hgt(r\evenx)>\even\hgt(\evenx)
\quad\text{and}\quad
\even\hgt(s_0r\evenx)<\even\hgt(s_0\evenx)
\]
so $r\evenx=s_0\evenx$ and thus $rx=x$ as required.

For QP2, suppose $r\in R(W)$, $x\in X$, $s\in S$ such that $\hgt(rx)>\hgt(x)$
and $\hgt(srx)<\hgt(sx)$.  If $\hgt(rx)-\hgt(x)$ is odd, then
$r\evenx=\even{rx}$, so $\even\hgt(r\evenx)>\even\hgt(\evenx)$.  If
$\hgt(rx)-\hgt(x)$ is even, then $r\evenx=s_0\even{rx}$, but
$\hgt(rx)-\hgt(x)\ge 2$, so again $\even\hgt(r\evenx)>\even\hgt(\evenx)$.
Similarly, $\even\hgt(s r \evenx)<\even\hgt(s \evenx)$.  It follows that
$r\evenx=s\evenx$, so $rx=sx$ as required.
\end{proof}

\begin{rem}
Note that if $X$ is already even, then $\evenX\cong X\times A_1$, where
$A_1$ is the regular representation of $A_1$.
\end{rem}

\begin{cor}
If $H\subset W$ is quasiparabolic, then so is its even subgroup $H\cap
W^0$.
\end{cor}

\begin{proof}
If $x_0\in W/H$ is the unique minimal element, then the minimal element
$\even x_0 =(x_0,\hgt(x_0))$ of the even double cover has stabilizer $H\cap W^0$.
\end{proof}

The following construction is extremely powerful.

\begin{thm}\label{thm:quotient}
Let $(X,\hgt)$ be a quasiparabolic $W\times W'$-set, and let $H\subset W'$
be a quasiparabolic subgroup such that every $H$-orbit in $X$ has height
bounded from below.  Let $X/H$ be the $W$-set of $H$-orbits in $X$, and
define $\hgt'(Hx)=\min_{y\in Hx}\{\hgt(y)\}$.  Then $(X/H,\hgt')$ is a
quasiparabolic $W$-set.
\end{thm}

\begin{proof}
Since we can write $X/H$ as $(X\times (W'/H))/\Delta(W')$, we can reduce
to the case that $H$ is a diagonal subgroup.  That is, $(X,\hgt)$ is a
quasiparabolic $W\times W'\times W'$-set, and we wish to show that
$X/\Delta(W')$, with the induced height function, is quasiparabolic.
This construction also commutes with taking the even double-cover, so we
may assume that $X$ is even.

We first need to verify that $\hgt'$ is a height function.  To see this,
let $x$ be a minimal representative of a $\Delta(W')$ orbit, and observe
that
\[
\hgt'(\Delta(W')sx)\le \hgt(sx) \le \hgt(x)+1 = \hgt(\Delta(W')x)+1.
\]
By symmetry, one also has
\[
\hgt'(\Delta(W')x)\le \hgt'(\Delta(W')sx)+1,
\]
and thus $(X/\Delta(W'),\hgt')$ is indeed a scaled $W$-set.  Note
furthermore that as we have assumed $X$ even, and every element of
$\Delta(W')$ has even length, $(X/\Delta(W'),\hgt')$ is even.  As a result,
property QP1 is immediate, and it remains only to verify property QP2.

Consider the set of quadruples $(x,s,r,w)$, $x\in X$, $s\in S$, $r\in
R(W\times W'\times W')$, $w\in W'$ such that
\begin{itemize}
\item[1.] $\hgt'(\Delta(W')x) = \hgt(x)$, i.e., $x$ is minimal in its
  orbit;
\item[2.] $\hgt'(\Delta(W')srx) = \hgt((w,w)srx)$;
\item[3.]
$
\hgt'(\Delta(W')sx)
=
\hgt'(\Delta(W')rx)
=
\hgt'(\Delta(W')srx)+1
=
\hgt'(\Delta(W')x)+1;
$
\item[4.] and the orbits $\Delta(W')sx$, $\Delta(W')rx$ are distinct.
\end{itemize}
Note that by Lemma \ref{lem:useful_lemma}, to verify property QP2, it will suffice to show
that no configuration of orbits as in the third and fourth conditions
exists, with $r\in R(W)$.  Since any such configuration extends to a
quadruple as above (simply choose a minimal representative $x$ and an
appropriate minimizing $w$), if we can show that the full set quadruples is
empty, this will imply property QP2.

Consider a quadruple which is minimal with respect to $\ell(w)$.  We have
the following two inequalities:
\begin{align}
\hgt(rx)\ge \hgt'(\Delta(W')rx) &= \hgt'(\Delta(W')x)+1=\hgt(x)+1\\
\hgt((w,w)sx)\ge \hgt'(\Delta(W')sx) &=
\hgt'(\Delta(W')srx)+1=\hgt((w,w)srx)+1.
\end{align}
We may thus apply the strong exchange condition for $(X,\hgt)$ to conclude
that $(w,w)srx$ can be obtained from $(w,w)sx$ by omitting a simple
reflection from some reduced word.  Since by assumption, $(w,w)srx\ne
(w,w)x$ (since the two orbits are distinct), we conclude that there exists
$w'$ obtained from $w$ by omitting a simple reflection such that
\[
(w,w)srx\in \{(w,w')sx,(w',w)sx\}.
\]
Without loss of generality, $(w,w)srx=(w,w')sx$.  Now, consider the new
quadruple $(x',s,r',w')$, where
\[
x' = (w,w)srx
\qquad r' = (w'w^{-1},1)\in R(W'\times W').
\]
This new quadruple satisfies conditions 1 through 4, since one readily
verifies that the four $\Delta(W')$ orbits have simply been pairwise
swapped.  Since $\ell(w')<\ell(w)$, this is a contradiction.
\end{proof}

\begin{rem}
Note that any minimal element of $X$ maps to a minimal element of $X/H$.
Also, observe that if $X$ is a quasiparabolic $W\times W'$-set, and $Y$ is
a quasiparabolic $W'\times W''$-set, then we obtain a quasiparabolic
$W\times W''$-set
\[
X\times_{W'}Y := (X\times Y)/\Delta(W');
\]
we then have
\[
X/H = X\times_{W'} (W'/H).
\]
\end{rem}

\begin{cor}\label{cor:induce_from_para}
Let $(W,S)$ be a Coxeter group, and let $W_I$ be a 
parabolic
subgroup of $W$.  If $H\subset W_I$ is a quasiparabolic subgroup of $W_I$,
then it is also quasiparabolic as a subgroup of $W$.  In particular, $W_I$
and its even subgroup are quasiparabolic in $W$.
\end{cor}

\begin{proof}
The transitive quasiparabolic $W\times W$-set $(W,\ell)$ restricts to a
quasiparabolic $W\times W_I$-set, and thus induces by the theorem a
transitive quasiparabolic $W$-set $W/H$.  The stabilizer of the minimal
element of this new set is precisely $H$, as required.

The remaining claim follows from the fact that for any Coxeter group $W$,
both $W$ and its even subgroup are quasiparabolic subgroups.
\end{proof}

\begin{rems}
  It follows from this construction that any coset $wG\in W/G$ has a unique
  decomposition of the form $wG=uvG$ with $u\in W^I$ (i.e., $u\in W$ is (right)
  $W_I$-minimal) and $\hgt(wG)=\ell(u)+\hgt(vG)$.  Existence follows by
  taking $w$ to be a minimal $G$-coset representative and utilizing the
  standard decomposition $w=uv$ with $u\in W^I$, $v\in W_I$; one then has
\[
\ell(u)+\ell(v) = \ell(w) = \hgt(wG) \le \ell(u)+\hgt(vG)\le
\ell(u)+\ell(v)
\]
so $\hgt(wG)=\ell(u)+\hgt(vG)$ as required.  To see uniqueness, observe that
the above construction represents $W/G$ as $W\times_{W_I} W_I/G$, or in
other words as the quotient of the set of pairs $(u,vG)$ by the action
$w\cdot (u,vG)=(uw^{-1},wvG)$ of $W_I$.  In each such orbit, there is a
unique choice of $w$ such that $uw^{-1}\in W^I$, and thus a unique orbit
representative of the desired form.  It follows that the Poincar\'e series
(see Section \ref{sec:hecke} below) of $W/G$ can be written as
\[
\PS_{W/G}(q)=\PS_{W/W_I}(q)\PS_{W_I/G}(q)
            =\frac{\PS_W(q)\PS_{W_I/G}(q)}{\PS_{W_I}(q)}.
\]
Equivalently, we have (compare Theorem \ref{thm:Poincare_divides} below)
\[
\frac{\PS_W(q)}{\PS_{W/G}(q)} = \frac{\PS_{W_I}(q)}{\PS_{W_I/G}(q)}.
\]
\end{rems}

\begin{rems}
  Experimentally (i.e., in every finite case which we have checked), there
  appears to be a partial converse to this statement, to wit that if
  quasiparabolic $H$ is contained in a conjugate parabolic subgroup, then
  it is a quasiparabolic subgroup of some standard parabolic subgroup.
  More precisely, we conjecture that the intersection of all conjugate
  parabolic subgroups containing $H$ is standard parabolic, and $H$ is a
  quasiparabolic subgroup of the intersection.
\end{rems}

\begin{defn} A {\em Coxeter homomorphism} $\phi:W\to W'$ is a group
  homomorphism such that $\phi(S)\subset S'\cup \{1\}$.
\end{defn}

\begin{rem}
Note that if $\phi:W\to W'$ is a Coxeter homomorphism, then
\[
\phi(R(W)) \subset R(W')\cup \{1\}.
\]
\end{rem}

\begin{cor}
Let $\phi:W\to W'$ be a Coxeter homomorphism.  If $H\subset W$ is
quasiparabolic, then so is $\phi(H)$; if $H'\subset W'$ is quasiparabolic,
then so is $\phi^{-1}(H')$.
\end{cor}

\begin{proof}
If $H'\subset W'$ is quasiparabolic, then there is a height-preserving
bijection $W'/H'\cong W/\phi^{-1}(H')$, and Definition \ref{defn:qp} is
immediate.

For the other direction, we observe that
\[
(1,\phi):W'\times W\to W'\times W'
\]
is also a Coxeter homomorphism, and thus the subgroup
\[
\Delta_\phi(W'):=(1,\phi)^{-1}(\Delta(W'))\subset W'\times W
\]
is quasiparabolic.  But then the set
\[
((W'\times W)/\Delta_\phi(W')) \times_{W} W/H
\]
is quasiparabolic, and readily verified to be transitive
such that $\phi(H)$ is the stabilizer of the minimal element. 
\end{proof}

We have shown above that if $H$ is an odd quasiparabolic subgroup of $W$,
then it contains a simple reflection, and its even subgroup $H^0$ is
quasiparabolic.  To show the converse, namely that any subgroup of $W$
containing a simple reflection and with quasiparabolic even subgroup is
itself quasiparabolic, it suffices to construct a suitable action of $A_1$
on $W/H^0$ giving an isomorphism $W/H^0\cong \even{W/H}$ of scaled $W\times
A_1$-sets and making $W/H^0$ quasiparabolic as a $W\times A_1$-set.

In fact, we have the following generalization of this fact.

\begin{thm}\label{thm:extend_by_simples}
Let $H\subset W$ be an even quasiparabolic subgroup, and let $I$ be the set
of all simple reflections normalizing $H$.  Then the extension of $W/H$ to
a $W\times W_I$-set by
\[
(w,w')\cdot gH:= w g (w')^{-1} H
\]
preserves quasiparabolicity.  In particular, $H W_I$ is also a
quasiparabolic subgroup of $W$.
\end{thm}

\begin{proof}
To see that $W/H$ is a {\em scaled} $W\times W_I$ set, we need to show that
$|\hgt(g s H)-\hgt(g H)|\le 1$ for $s\in I$.  If $g$ has minimum length in
its coset, then we find that
\[
\hgt(g s H)\le \ell(g s)\le \ell(g)+1\le \hgt(g H)+1,
\]
and thus the action of $(1,s)$ can never increase the height by more than
1; thus its inverse can never {\em decrease} the height by more than 1, and
since $s^2=1$, we are done.  Moreover, since $H$ is even, we in fact find
that $W/H$ is an {\em even} scaled $W\times W_I$ set.

It thus remains to show that property QP2 holds.
There are, as before, four cases to consider:
\begin{align}
\hgt(r g H)=\hgt(s g H)=\hgt(g H)+1=\hgt(s r g H)+1&\quad r\in R(W),s\in S\\
\hgt(g r H)=\hgt(s g H)=\hgt(g H)+1=\hgt(s g r H)+1&\quad r\in R(W_I),s\in S\\
\hgt(r g H)=\hgt(g s H)=\hgt(g H)+1=\hgt(r g s H)+1&\quad r\in R(W),s\in I\\
\hgt(g r H)=\hgt(g s H)=\hgt(g H)+1=\hgt(g r s H)+1&\quad r\in R(W_I),s\in I.
\end{align}
The second case is a special case of the first (replacing $r$ by $g r
g^{-1}$), and similarly the fourth is a special case of the third.
Moreover, the first case is just property QP2 as a $W$-set, so is immediate.

We are thus left with the case $r\in W$, $s\in I$.  Now, suppose $g$ is a
shortest element of $g H$, and choose a reduced word $g=s_1s_2\cdots s_k$.
Then
\[
\hgt(r s_1 s_2\cdots s_k s H)<\hgt(s_1 s_2\cdots s_k s H),
\]
so by Corollary \ref{cor:strong_exch}, either
\[
r s_1 s_2\cdots s_k s H = s_1 s_2\cdots s_k H
\]
(and we are done), or there exists $1\le l\le k$ such that
\[
r s_1 s_2\cdots s_k s H
=
s_1 s_2 s_{l-1} s_{l+1}\cdots s_k s H.
\]
But then
\[
r s_1 s_2\cdots s_k H
=
s_1 s_2 s_{l-1} s_{l+1}\cdots s_k H,
\]
since $s$ normalizes $H$, 
so that
\[
\hgt(r s_1 s_2\cdots s_k H)\le k-1=\hgt(g H)-1,
\]
a contradiction.

The final claim is a special case of the following corollary.
\end{proof}

\begin{cor}
If $H$, $I$ are as above, and $K$ is any quasiparabolic subgroup of $W_I$,
then $HK$ is a quasiparabolic subgroup of $W$.
\end{cor}

\begin{proof}
Consider the quasiparabolic $W$-set $W/H\times_{W_I} W_I/K$.
\end{proof}

\begin{cor}
If the subgroup $H\subset W$ contains a simple reflection, then $H$ is
quasiparabolic iff $H\cap W^0$ is quasiparabolic.
\end{cor}

A given odd quasiparabolic subgroup $G \subset W$ can in general be obtained in more than
one way via the construction of Theorem \ref{thm:extend_by_simples}.  The
smaller the normal subgroup $H$ of $G$ being used in the construction, the more 
other odd subgroups it explains.
 We would thus in particular like to understand the
minimal such subgroup (if it exists).  There is, in fact, a natural
candidate for this minimal normal subgroup, namely the subgroup generated
by all elements that are forced to be in the stabilizer of the minimal
element by property QP2 alone.  Although we cannot as yet prove that this
subgroup is quasiparabolic, we can at least prove the following.

\begin{thm}\label{thm:natural_normal}
Let $H$ be a quasiparabolic subgroup of $W$, and let $N$ be the subgroup of
$H$ generated by all elements of the form $w^{-1}srw$ for $w\in W$, $s\in
S$, $r\in R(W)$ such that $\hgt(rwH)\ge \hgt(wH)$, $\hgt(srwH)\le
\hgt(swH)$.  Then $N$ is normal, and $H/N$ is generated by simple
reflections.
\end{thm}

\begin{proof}
Since $H$ is quasiparabolic, the conditions on $w$, $s$, $r$ force
$rwH=swH$ and thus $w^{-1}srw\in H$ as required.  Moreover, for any $h\in
H$, $(wh,s,r)$ also satisfies the conditions; it follows that $N$ is indeed
a normal subgroup of $H$.

Now, let $s_1s_2\dots s_k=h$ be an arbitrary reduced word multiplying to an
element of $H$, and consider the sequence
\[
0=\hgt(H),\ \hgt(s_k H),\ \hgt(s_{k-1}s_k H),\dots,\ \hgt(s_1s_2\cdots s_k H)=0
\]
of heights.  For each $j$ such that
\[
\hgt(s_js_{j+1}\cdots s_k H)=\hgt(s_{j+1}\cdots s_k H),
\]
choose a reduced expression
\[
t_{j1}t_{j2}\cdots t_{jn} H=s_js_{j+1}\cdots s_k H=s_{j+1}\cdots s_k H,
\]
and extend the given word by inserting
\[
t_{j1}t_{j2}\cdots t_{jn} t_{jn}\cdots t_{j2}t_{j1}
\]
before and after $s_j$.  If we then break the word between each pair
$t_{jn}t_{jn}$, we obtain a factorization of $h$ as a product of words each
having at most one step in which the height does not change.

Let $s_1s_2\cdots s_k$ thus be such a word, and suppose first that at no
step does the height remain unchanged.  We claim that, in that case,
$s_1s_2\cdots s_k\in N$.  Suppose otherwise, and choose a counterexample
of minimum length.  Since the height must increase at the first step, and
eventually decreases back to 0, there exists $j$ such that
\[
\hgt(s_js_{j+1}\cdots s_k H)=\hgt(s_{j+1}\cdots s_k H)-1.
\]
Choose the largest such $j$.  By Corollary \ref{cor:strong_exch}, we find
\[
s_j s_{j+1}\cdots s_k H = s_{j+1}\cdots s_{l-1} s_{l+1}\cdots s_k H
\]
for some $l$, and by the proof of that corollary, it follows that
\[
(s_{l+1}\cdots s_k)^{-1} s_l\cdots s_j\cdots s_{l-1} (s_{l+1}\cdots s_k)
\in
N.
\]
But then
\[
s_1s_2\cdots s_{j-1}s_{j+1}\cdots s_{l-1}s_{l+1}\cdots s_k
\in
h N
\]
would be a shorter counterexample.

Similarly, let $s_1s_2\cdots s_k=h\in H$ be a word such that at
precisely one step the height remains the same (necessarily of odd length,
by parity considerations); we claim that
\[
h\in s N
\]
for some simple reflection $s$.  Decompose the word for $h$ as $h=vtw$,
$v,w\in W$, $t\in S$, such that $twH=wH$.  If $\hgt(wH)=0$, then $t\in H$
and $v,w\in N$, so we are done. Otherwise, there exists a simple reflection
$u\in S$ such that $\hgt(uwH)=\hgt(wH)-1$.  Consider the $\langle
t,u\rangle$-orbit generated by $wH$.  By quasiparabolicity, this orbit has
a unique minimal element $w'H$, such that $t'w'H=w'H$,
$\hgt(u'w'H)>\hgt(w'H)$, with $\{t',u'\}=\{t,u\}$ (which depends
on the parity of $\hgt(wH)-\hgt(w'H)$).  We can freely replace $v^{-1}$ and
$w$ by any words in the same coset such that the height changes at each
step, and may therefore assume that
\[
v^{-1}=w =
\begin{cases}
(t'u')^{(\hgt(w H)-\hgt(H))/2} w' & \hgt(wH)\equiv \hgt(H)\pmod{2}\\
u' (t'u')^{(\hgt(w H)-\hgt(H)-1)/2} w' & \hgt(wH)\equiv \hgt(H)+1\pmod{2}.
\end{cases}
\]
But then
\[
v t w = (w')^{-1} u' (t' u')^{\hgt(wH)-\hgt(H)-1} w'.
\]
Since
\[
\hgt(u' w' H)=\hgt(u' (t' u')^{\hgt(wH)-\hgt(H)} w' H)
=
\hgt(w' H)+1=\hgt((t' u')^{\hgt(wH)-\hgt(H)} w' H)+1,
\]
we find
\[
(w')^{-1} (t' u')^{\hgt(wH)-\hgt(H)} w'\in N,
\]
so that
\[
v t w\in (w')^{-1} t' w' N,
\]
giving us the desired result by induction (since $\hgt(w' H)<\hgt(wH)$).
\end{proof}

\begin{conj}\label{conj:nice_normal_is_qp}
For any quasiparabolic subgroup $H\subset W$, the normal subgroup of
Theorem \ref{thm:natural_normal} is also quasiparabolic.
\end{conj}

\section{Perfect involutions}\label{sec:inv}
One of our original motivating examples of a quasiparabolic $W$-set is the
set of fixed-point free involutions in $S_{2n}$, with height function
$(\ell(\iota)-n)/2$.  This generalizes considerably, to Theorem
\ref{thm:fpf} below.  We first need to introduce some notation.

Let $W$ be a Coxeter group, and let $W^+$ be the semidirect product of $W$
by the group of permutations of $S$ that induce Coxeter automorphisms of
$W$; this inherits a length function from $W$ by taking the length of such
a Coxeter automorphism to be 0; equivalently, $W^+$ acts on the set of
roots of $W$, and the length function counts (as usual) the number of
positive roots taken to negative roots by the given element of $W^+$.

\begin{lem}
For any Coxeter group $W$, $(W^+,\ell)$ is a quasi-parabolic $W\times
W$-set.
\end{lem}

\begin{proof}
Indeed, each $W\times W$-orbit in $W^+$ has a unique minimal element,
namely the associated Coxeter automorphism $\phi:W\to W$; the stabilizer of
that minimal element is the quasiparabolic subgroup $\Delta_\phi(W)$.
\end{proof}

\begin{defn}
An involution $\iota\in W^+$ is {\em
  perfect} if for all $r\in R(W)$, $r$ commutes with $\iota r\iota$.
We will denote by ${\cal I}$ the set of all perfect involutions.
\end{defn}

\begin{rems}
Note that $r$ commutes with $\iota r\iota$ iff $(r\iota)^4=1$.  In the case
$W=S_n$, any fixed point free involution is perfect, as follows easily from
the fact that reflections are just 2-cycles.  Similarly, any element
conjugate to the diagram automorphism of $S_{2n}$ is perfect for precisely the
same reason.  These two classes of perfect involutions will give rise to
(quasiparabolic) scaled $W$-sets with negated heights.
\end{rems}

\begin{rems}
  This appears to be too stringent a condition when $W$ is infinite; for
  instance, the obvious analogue in $\tilde{A}_{2n-1}$ of the case of
  fixed-point-free involutions in $A_{2n-1}$ (i.e., one of the two
  conjugacy classes of preimages in $\tilde{A}_{2n-1}$) do not give perfect
  involutions, but are sufficiently well-behaved that they very likely form
  a quasiparabolic set.
\end{rems}

\begin{lem}\label{lem:fpf_prop1}
If $\iota\in W^+$ is perfect, then for all $r\in R(W)$, if
\[
\ell(\iota),\ell(r\iota r)<\ell(r\iota)=\ell(\iota r)
\]
or
\[
\ell(\iota),\ell(r\iota r)>\ell(r\iota)=\ell(\iota r),
\]
then $\iota=r\iota r$.
\end{lem}

\begin{proof}
Note that since $\iota$ is an involution, $r\iota = (\iota r)^{-1}$, and
thus the two elements have the same length as stated.  Note also that the
contrapositive of the lemma reads that if $r\iota$ has order 4, then
$\ell(r\iota)$ is between $\ell(\iota)$ and $\ell(r\iota r)$, and this
inclusion is strict by parity considerations.

Thus suppose that $r\iota$ has order 4.  We can write $r=s_\alpha$ for some
positive root $\alpha$; the fact that $r\iota$ has order 4 implies that
$\iota(\alpha)$ is orthogonal to $\alpha$.  But then
\begin{align}
\ell(\iota)<\ell(\iota r) &\iff \iota(\alpha)>0\\
                          &\iff r\iota(\alpha)>0\\
                          &\iff \ell(r\iota)<\ell(r\iota r),
\end{align}
as required, and similarly for the opposite inequalities.
\end{proof}

The set of perfect involutions is certainly acted on by $W$ by
conjugation; it very nearly becomes a scaled $W$-set by setting
$\hgt(\iota)=\ell(\iota)/2$; this fails only in that on some orbits the
heights might lie in $1/2+\Z$ rather than $\Z$, but this has no effect on
the theory.  (Indeed, for the above theory of quasiparabolic sets to work,
we need only that on each orbit, the height function lies in a fixed coset
of $\Z$.)

\begin{thm}\label{thm:fpf}
The set ${\cal I}$ of perfect involutions in $W^+$, together with the above
height function, is a quasiparabolic $W$-set.
\end{thm}

\begin{proof}
To show property QP1, let $\iota$ be a perfect involution, and
let $r\in R(W)$ be such that $\hgt(r\cdot \iota)=\hgt(\iota)$;
equivalently, $\ell(r\iota r)=\ell(\iota)$.  But then Lemma
\ref{lem:fpf_prop1} implies that $r\iota r=\iota$ as required.

It remains to show property QP2; let, therefore, $\iota$ be a perfect
involution, and $r\in R(W)$, $s\in S$ such that (recalling Lemma
\ref{lem:useful_lemma})
\[
\hgt(r\cdot \iota)=\hgt(s\cdot \iota)=\hgt(\iota)+1=\hgt(sr\cdot\iota)+1,
\]
or equivalently
\[
\ell(r\iota r)=\ell(s\iota s)=\ell(\iota)+2=\ell(sr\iota rs)+2.
\]
We need to show that $r\iota r=s\iota s$.

We first observe that
\[
\ell(r\iota)=\ell(\iota)+1.
\]
Indeed, by Lemma \ref{lem:fpf_prop1}, $\ell(r\iota)$ is between
$\ell(\iota)$ and $\ell(r\iota r)$, and thus (since the latter differ by
2), must be $\ell(\iota)+1$ as required.  Similarly,
\[
\ell(s\iota)=\ell(\iota)+1.
\]
Moreover, since $sr\iota rs$ is also a perfect involution, and $srs$ is a
reflection, we find that
\[
\ell(sr\iota rs s)=\ell(sr\iota rs (srs))=\ell(sr\iota rs)+1=\ell(\iota)+1,
\]
or in other words,
\[
\ell(sr\iota r)=\ell(sr\iota s)=\ell(\iota)+1.
\]

Now, consider the element $sr\iota$.  Since $\ell(r\iota)=\ell(\iota)+1$
and $s$ is a simple reflection, we find that $\ell(sr\iota)-\ell(\iota)\in
\{0,2\}$.  We consider the two cases separately.

If $\ell(sr\iota)=\ell(\iota)$, then by the fact that $W^+$ is
quasiparabolic as a $W\times W$-set and the fact that
$\ell(s\iota)=\ell(r\iota)=\ell(\iota)+1$, we conclude that
$s\iota=r\iota$.  But $W^+$ is a group, so we can cancel $\iota$ to find
$s=r$, and thus in particular $s\iota s=r\iota r$ as required.

If $\ell(sr\iota)=\ell(\iota)+2$, then on the one hand
\[
\ell(sr\iota)=\ell(r\iota r)=\ell(r\iota)+1=\ell(sr\iota r)+1,
\]
so $sr\iota=r\iota r$, while on the other hand
\[
\ell((srs)s\iota)=\ell(s\iota s)=\ell(s\iota)+1=\ell((srs)s\iota s)+1,
\]
so $sr\iota = (srs)s\iota = s\iota s$.  Therefore $r\iota r=s\iota s$
as required.
\end{proof}

\begin{rem}
In addition to the special case of fixed-point-free involutions in
$S_{2n}$, there are two general instances of perfect involutions: the
identity element is always perfect, as is the diagram automorphism of
$W\times W$ that swaps the two factors.  The latter quasiparabolic $W\times
W$ set turns out to be naturally isomorphic to $W$.
\end{rem}

Suppose $\iota$ is a $W$-minimal perfect involution, so that by the above
result, the centralizer $Z_W(\iota)$ is quasiparabolic.  Let $z_W(\iota)$
denote the normal subgroup attached to this group by Theorem
\ref{thm:natural_normal}.  This has a particularly nice description
directly in terms of $\iota$.  First a useful lemma.

\begin{lem}
Let $\iota$ be a perfect involution, and let $r$, $r'$ be reflections such
that $r\iota r=r'\iota r'\ne \iota$.  Then $r'\in \{r,\iota r \iota\}$.
\end{lem}

\begin{proof}
Rewrite the hypothesis as $(r\iota)^2=(r'\iota)^2\ne 1$, and work in the
reflection representation of $W^+$.  Let $r=r_\alpha$ for some 
 root
$\alpha$ normalized so that $\alpha\cdot\alpha=2$.  Since $(r\iota)^2\ne
1$, it follows that $\iota\alpha\cdot\alpha=0$, and thus
\[
(r\iota)^2 = r_\alpha r_{\iota \alpha}
           = r_{\alpha_+} r_{\alpha_-},
\]
where $\alpha_{\pm}:=(1\pm \iota)\alpha/2$.
Now, the vectors $\alpha_{\pm}$ are uniquely determined up to sign
as unit vectors which are simultaneous eigenvectors of $\iota$ and
$(r\iota)^2$, and thus if $r'=r_{\alpha'}$, we find
\[
\alpha'_+ = \pm \alpha_+\qquad
\alpha'_- = \pm \alpha_-.
\]
Solving for $\alpha'$, we find
\[
\alpha'\in \{\pm \alpha,\pm \iota \alpha\}.
\]
The claim follows.
\end{proof}

\begin{thm}
Let $\iota \in {\cal I}$.
The group $z_W(\iota)$ is the subgroup of $W$ generated by elements of the
form $(r\iota)^2$ for $r\in R(W)$.
\end{thm}

\begin{proof}
We first observe that if $w\in W$, $s\in S$, $r\in R(W)$ are such that
\[
\hgt(sw\cdot \iota)=\hgt(rw\cdot \iota)=\hgt(w\cdot \iota)+1=\hgt(srw\cdot \iota)+1,
\]
so that $w^{-1} srw\in z_W(\iota)$, then
\[
w^{-1}sw\cdot\iota=w^{-1}rw\cdot\iota\ne \iota,
\]
and thus (assuming $r\ne s$) $w^{-1}rw = \iota w^{-1}sw\iota$, so that
\[
w^{-1}srw = (w^{-1}sw\iota)^2,
\]
and thus by the lemma $z_W(\iota)$ is generated by elements of the desired
form.

It remains to show that for any reflection, $(r\iota)^2\in z_W(\iota)$.
Write $r=w^{-1}sw$, and define $r':=w\iota r\iota w^{-1}$, so that $sw\cdot
\iota = r'w\cdot\iota$.  If $\hgt(sw\cdot\iota)>\hgt(w\cdot\iota)$, then
then
\[
\hgt(sw\cdot\iota)
=
\hgt(r'w\cdot\iota)
=
\hgt(w\cdot\iota)+1
=
\hgt(sr'w\cdot\iota)+1,
\]
so that $(r\iota)^2=w^{-1} r' s w\in z_W(\iota)$.  If $\hgt(sw\cdot\iota)<\hgt(w\cdot\iota)$, then
\[
\hgt(s(r'w)\cdot\iota)
=
\hgt(r'(r'w)\cdot\iota)
=
\hgt(r'w\cdot\iota)+1
=
\hgt(sr'(r'w)\cdot\iota)+1,
\]
so that $(r\iota)^2=(\iota r)^2=w^{-1} r' s r' (r' w)\in z_W(\iota)$.
Finally, if $\hgt(sw\cdot\iota)=\hgt(w\cdot\iota)$, then
\[
r\cdot\iota = \iota,
\]
so that $(r\iota)^2=1\in z_W(\iota)$.
\end{proof}

Naturally, Conjecture \ref{conj:nice_normal_is_qp} would imply that the
groups $z_W(\iota)$ are quasiparabolic.  It is fairly straightforward to
classify the pairs $(W,\iota)$ with $W$ finite and $\iota$ perfect.  These
include a number of sporadic cases (for which we have verified
quasiparabolicity via computer), as well as infinite families coming from
fixed-point-free involutions in $A_{2n-1}$ and a corresponding case in
$D_{2n}$.  For the former, $z_W(\iota)$ is simply the even subgroup of
$Z_W(\iota)$ and therefore its quasiparabolicity is automatic.  Thus only
the $D_{2n}$ case (which we have verified through $D_{12}$) remains open
among finite cases.
\section{Bruhat order}\label{sec:bruhat}

\begin{defn}
  Let $(X, \hgt)$ be a quasiparabolic $W$-set.  The {\em Bruhat order} on $X$ is the weakest
  partial order such that for $x\in X$, $r\in R(W)$, $x\le rx$ iff
  $\hgt(x)\le \hgt(rx)$.
\end{defn}

\begin{rem}
Note that as with the usual Bruhat order, $x$ and $rx$ are always
comparable, and $x\le y$ implies $\hgt(x)\le \hgt(y)$.
\end{rem}

Property QP2 can be rephrased in terms of the Bruhat order:

\begin{prop}\label{prop:QP2_bruhat}
If $x<y$ and $sy<sx$, then $sy=x$.
\end{prop}

\begin{proof}
Since $x<y$ implies $\hgt(x)<\hgt(y)$, we conclude that $\hgt(y)=\hgt(x)+1$.
But then $y$ must cover $x$ in the Bruhat order, so that $y=rx$ for some
reflection $r$.  The proposition is then precisely QP2.
\end{proof}

\begin{rem}
We had originally developed a theory in which instead of QP1 and QP2, we
instead insisted only that every orbit of a parabolic subgroup should have
at most one minimal and at most one maximal element.  One can show that
this condition need only be checked in rank 2, where it is equivalent to
quasiparabolicity.  When we eventually began considering how to extend
Bruhat order to such sets, we discovered that there were essentially only
two instances (see Example \ref{eg:bozo} below) of
``quasiparabolic'' sets in which such an
extension failed to exist, so we decided a change in definition was in order. 
In looking at the proofs in \cite{HumphreysJE:1990}
concerning Bruhat order, we found that most of the arguments relied only on
the claim of this proposition.  Rewriting in terms of reflections gave QP1
and QP2.
\end{rem}

There is a related reformulation of quasiparabolicity.

\begin{prop}\label{prop:QP_via_Bruhat}
Let $(X, \hgt)$ be a scaled $W$-set.  Then $X$ is quasiparabolic iff there exists a
partial ordering $\le$ on $X$ such that
\begin{itemize}
\item[(1)] $\hgt$ is strictly increasing: if $x<y$ then $\hgt(x)<\hgt(y)$.
\item[(2)] For any $x\in X$, $r\in R(W)$, $x$ and $rx$ are comparable.
\item[(3)] For any $s\in S$, if $x<y$, $sy<sx$, then $x=sy$.
\end{itemize}
\end{prop}

\begin{proof}
If $X$ is quasiparabolic, we have shown that Bruhat order satisfies these
three properties.  Conversely, QP1 follows by observing that if
$\hgt(rx)=\hgt(x)$, then the only way $x$ and $rx$ can be comparable is if
they are equal, while QP2 follows by observing that the hypothesis of QP2
implies via comparability that $x<rx$, $sx>srx$ and thus $x=srx$ as
required.
\end{proof}

\begin{rem}
  This gives a possible strategy for proving quasiparabolicity of conjugacy
  classes of involutions in $W^+$ for $W$ infinite, namely consider the
  restriction of the usual Bruhat order to the given conjugacy class.  The
  first condition of the proposition always holds, since length is strictly
  increasing in ordinary Bruhat order.  The third condition follows by
  observing that since $|\ell(s\iota s)-\ell(\iota)|\le 2$ and is 0 if and
  only if $s\iota=\iota s$, then the only way the ordering of $\iota$ and
  $\iota'$ can be reversed is if their lengths differ by 2 and conjugation
  by $s$ swaps their lengths.  But then $\iota<s\iota$, so by Lemma
  \ref{lem:bruhat} (or, rather, the statement for ordinary Bruhat order
  that this lemma generalizes), $\iota\le s\iota'$, and $\iota<\iota s$ so
  $\iota\le s\iota' s$.  Since the lengths agree, $\iota=s\iota's$ as
  required.  Thus only the second condition of the proposition
  (comparability of $\iota$ and $r\iota r$) need be shown.  Note that it
  follows from Lemma \ref{lem:fpf_prop1} that $\iota$ and $r \iota r$ are
  comparable whenever $r\iota$ has order 4.  Checking this for the classes
  of interest in $\tilde{A}_{2n-1}$ or $\tilde{C}_{2n}$ (where nice
  combinatorial characterizations of Bruhat order are known
  \cite{BjornerA/BrentiF:2005}) appears tractable, but quite technical, in
  the absence of further ideas.
\end{rem}

We will also need a slight variant of Proposition \ref{prop:QP2_bruhat}.

\begin{lem}\label{lem:bruhat}
If $x\le y$, $s \in S$  then
\begin{itemize}
\item[(1)] Either $sx\le y$ or $sx\le sy$;
\item[(2)] Either $x\le sy$ or $sx\le sy$.
\end{itemize}
In particular, if $sy\le y$ then $sx\le y$; if $x\le sx$, then $x\le sy$.
\end{lem}

\begin{proof}
  By symmetry, we need only prove $(1)$; the corresponding special case is
  immediate.  Following Humphreys \cite[Prop. 5.9]{HumphreysJE:1990}, we reduce to
  the case $y=rx$ for some $r\in R(W)$.  If $r=s$ or $sx\le x$, then the
  Lemma is immediate.  So we may assume $x<sx$.  Then $sx$ and $sy=(srs)sx$
  are comparable, and the previous proposition gives the desired dichotomy.
\end{proof}

Note as a special case, one has $\min\{x,sx\}\le \min\{y,sy\}$ whenever
$x\le y$.  This in turn is a special case of the following Proposition.
If $I\subset S$, $x\in X$ are such that the orbit $W_Ix$ is bounded from below,
let
\[
\pi_I(x) :=\min_{w\in W_I} wx
\]
denote the (unique) minimal element of that orbit.

\begin{prop} \label{prop:proj_para_bruhat}
Let $(X, \hgt)$ be a quasiparabolic $W$-set.
 Suppose $I\subset S$ is such that every $W_I$-orbit is bounded from below.
Then the map $\pi_I:X\to X$ is order-preserving.
\end{prop}

\begin{proof} We have already seen that this holds when $|I|=1$. Thus if
$s_1$,\dots,$s_k$ are the simple reflections in $I$ in some order, the
composition
\[
\pi_{\langle s_1\rangle}\circ\cdots \circ \pi_{\langle s_k\rangle}
\]
is order-preserving and height-decreasing. If we iterate this operation starting at
$x \in X$, we will eventually arrive at a fixed point
clearly equal to $\pi_I(x)$; the proposition follows.
\end{proof}

In particular, for any set of parabolic subgroups $W_{I_1},\dots,W_{I_k}$,
one has the implication
\[
x\le y\implies \pi_{I_i}(y)\le \pi_{I_i}(y), 1\le i\le k.
\]
In the case of $W$ viewed as a quasiparabolic $W$-set, it is a well-known
result of Deodhar that the converse holds whenever $\cap_i I_i=\emptyset$.
This does not hold for general quasiparabolic $W$-sets, however. For
instance, if one views $W$ as a $W\times W$-set, then every maximal
parabolic of $W\times W$ is transitive, so the projected orderings provide
no information. (An even worse example is $W$ acting on $W/W^0$, since
then every nontrivial parabolic subgroup is transitive.) One must thus at
the very least add the necessary condition that it should be possible to
reconstruct $x$ from the projections $\pi_{I_i}(x)$.
We do not know of any examples in which this necessary condition fails to be sufficient.

Let $[x,\infty)$ denote the subset of $X$ consisting of elements $y\ge x$,
and similarly for $(-\infty,y]$.

\begin{cor}\label{cor:easy_balanced}
  Let $x,y\in X$, $s\in S$ be such that $x\le sx$ and $sy\le y$.  Then the
  intervals $[x,\infty)$ and $(-\infty,y]$ are preserved by $s$.  If we
  further have $x\le y$, then $[x,y]$ is also preserved by $s$.
\end{cor}

\begin{proof}
If $x\le z$, then since $x\le sx$, the lemma gives $x\le sz$; similarly, if
$z\le y$, then since $sy\le y$, the lemma gives $sz\le y$.
\end{proof}

It will be useful to be able to restrict our attention to {\em even}
quasiparabolic  $W$-sets, so we will need to know how the Bruhat orders on $X$ and
$\evenX$ are related.

\begin{lem}\label{lem:easy_even}
  If $x\le y$, then $\evenx\le \eveny, \, s_0\evenx\le s_0\eveny$.
  Conversely, if $(x,k)\le (y,l)$ for 
some $k,l \in \F_2$, then $x\le y$.
\end{lem}

\begin{proof}
By transitivity, we reduce to the case $y=rx$, $\hgt(y)>\hgt(x)$.  But then
$r\evenx\in \{\eveny,s_0\eveny\}$, so that
\[
\even\hgt(r\evenx)\ge \even\hgt(\eveny)=\hgt(y)>\hgt(x)=\even\hgt(\evenx).
\]
This implies that either $\evenx\le \eveny$ or $\evenx\le s_0\eveny$.
Since $\evenx<s_0\evenx$ it follows by Lemma \ref{lem:bruhat} that
$\evenx\le \eveny$, and since $\eveny<s_0\eveny$, that $s_0\evenx\le
s_0\eveny$.

For the converse, we may again reduce to the case $(y,l)=r(x,k)=(rx,k+1)$
with $r\in R(W\times A_1)$ and $\hgt(y,l)>\hgt(x,k)$.  If $r=s_0$, then
$y=x$ and we are done, so we may assume that $r\in R(W)$.  Since $x$ and
$y=rx$ are comparable, it remains only to rule out the possibility that
$y<x$.  But then $\eveny<\evenx$ by the previous case, and thus (since the
ordering of the heights changes) we must have $(x,k)=s_0\evenx$ and
$(y,l)=s_0\eveny$.  In other words, we have both $\eveny<\evenx$ and
$s_0\evenx<s_0\eveny$.  By Proposition \ref{prop:QP2_bruhat}, it follows that
$s_0\evenx=\eveny$, but this is impossible.
\end{proof}

Determining if $s_0\evenx\le \eveny$ is somewhat more subtle.

\begin{lem}\label{lem:hard_even}
If $x<y$, then either $s_0\evenx$ and $\eveny$ are incomparable or
$s_0\evenx<\eveny$, with the latter occurring iff there exists a chain
\[
x=x_0<x_1=r_1x_0<x_2=r_2x_1<\cdots<x_k=y,
\]
$r_i \in R(W)$,
which at some point increases the height by an even amount.
\end{lem}

\begin{proof}
  We first note that if $x<rx=y$ with $\hgt(rx)-\hgt(x)$ even, then
  $r\evenx=s_0\eveny$, so that $s_0\evenx = r\eveny$, and thus $s_0\evenx$
  and $\eveny$ are comparable, and comparing heights shows that
  $s_0\evenx<\eveny$ as desired.  But then given any chain as hypothesized,
  with say the $k$-th step even, we may apply the previous lemma to
  conclude
\[
s_0\evenx<s_0\even{x_1}<s_0\even{x_2}<\cdots
  <s_0\even{x_{k-1}}<\even{x_k}<\even{x_{k+1}}<\cdots<\eveny.
\]

Conversely, suppose $s_0\evenx<\eveny$, and consider a chain
\[
s_0\evenx<r_1s_0\evenx<\cdots<\eveny.
\]
Each element in this chain is either $A_1$-minimal or $A_1$-maximal; since
the chain begins with a maximal element and ends with a minimal element,
there is at least one step going from maximal to minimal.  Now, the image
of the chain in $\evenX$ is a chain in $X$, except that some consecutive
elements may agree.  If we remove these elements, however, we obtain a
valid chain.  It therefore suffices to show that any step going from
maximal to minimal maps to a step increasing the height by an even amount.
Note that such a step goes from $s_0\even{u}$ to $\even{v}$ for elements
$u$ and $v$ which are related by a reflection in $W$ since $s_0$ 
decreases the height of $s_0\even{u}$.  Since the heights in $\evenX$ differ by an odd
amount, the heights in $X$ differ by an even amount, and $u\ne v$, since
again the height must increase.
\end{proof}

Since the restriction $X|_{W_I}$ of a quasiparabolic $W$-set $X$ to a
parabolic subgroup $W_I$ is quasiparabolic, it is natural to ask how the
Bruhat orders compare.  Clearly if two elements are in distinct
$W_I$-orbits, then they are incomparable with respect to the Bruhat order
on $X|_{W_I}$.  Otherwise, we conjecture in general that the Bruhat order
of the restriction agrees with the restriction of the Bruhat order to any
given orbit.  Unfortunately, to date we can only prove this with an
additional boundedness assumption.

\begin{thm}\label{thm:bruhat_para}
Let $(X, \hgt)$ be a quasiparabolic $W$-set, and let $W_I\subset W$ be a 
parabolic subgroup.  If $x,y\in X$ are in the same bounded $W_I$-orbit,
then $x\le y$ in $X$ iff $x\le y$ in $X|_{W_I}$.
\end{thm}

\begin{proof}
  Since $R(W_I)\subset R(W)$, one direction is obvious (even without
  assuming boundedness).  Thus assume that $x\le y$ in $X$.  By symmetry,
  we may assume that $W_Ix=W_Iy$ has a minimal element $x_0$.  Moreover, by
  transitivity, we may assume that $y=rx$ for some $r\in R(W)$, and it will
  suffice to show that $y=r'x$ for some $r'\in R(W_I)$.  Let $y=wx_0$ be a
  reduced expression for $y$ inside $X|_{W_I}$.  Then we have
\[
\hgt(w^{-1}rwx_0) = \hgt(w^{-1}x)\ge \hgt(x_0)
\]
while
\[
\hgt(w(w^{-1}rw)x_0)=\hgt(x)<\hgt(y)=\hgt(wx_0).
\]
It follows by the strong exchange condition that $rwx_0=w'x_0$ where $w'$
is obtained from a reduced word representing $w$ by removing a single
simple reflection (which must be in $W_I$ since $w\in W_I$).  It follows
that
\[
y = w w^{\prime{-1}} x,
\]
and the result follows by observing that $w w^{\prime{-1}}\in R(W_I)$.
\end{proof}

\begin{rems}
In particular, if $X$ is transitive as a $W_I$-set, then restricting to
$W_I$ does not change the Bruhat order.  As a further special case, the
Bruhat order of $W$ viewed as a $W\times W$-set agrees with the usual
Bruhat order.
\end{rems}

\begin{rems}
Restriction is a special case of the operation $\times_W$ described above;
it would be desirable to understand how to construct the Bruhat order on such
a product from the two original Bruhat orders.
\end{rems}

The proofs of our remaining results on Bruhat order require a global
assumption of boundedness, though we conjecture in each case that this
assumption is unnecessary.  In addition to enabling us to perform induction
on heights, boundedness also gives us an alternate characterization of the
Bruhat order, analogous to the interpretation of the usual Bruhat order in
terms of subwords.  As a result, many of the proofs in the literature carry
over with only minor changes required.

Below, we will fix a quasiparabolic $W$-set $(X, \hgt)$ and a $W$-minimal element
$x_0 \in X$; to simplify notations, we will assume that $X$ is transitive and
$x_0$ has height 0.  (Since elements in different orbits are incomparable,
there is no loss in assuming transitivity.)

\begin{thm}
Let $y=s_1\cdots s_k x_0$ be a reduced expression.  Then $x\le y$ iff one
can write
\[
x = s_{i_1}\cdots s_{i_j}x_0
\]
for some $1\le i_1<i_2<\cdots<i_j\le k$.
\end{thm}

\begin{proof}
  The proof of \cite[Thm
  5.10]{HumphreysJE:1990} carries over mutatis mutandum, except that one must
  multiply by reflections on the left and substitute Lemma \ref{lem:bruhat}
  for \cite[Prop. 5.9]{HumphreysJE:1990}
\end{proof}

Similarly, Proposition 5.10 of Humphreys carries over immediately.

\begin{prop}
If $x<y$ and there is no $z$ such that $x<z<y$, then $\hgt(y)=\hgt(x)+1$.
In particular, any saturated chain
\[
x=x_0<x_1<x_2<\cdots<x_k=y
\]
has length $k=\hgt(y)-\hgt(x)$, and the poset $(X,\le)$ is graded.
\end{prop}

In principle, there are two natural partial orderings on the set of perfect
involutions in $W^+$, namely the Bruhat order $\le_{\cal I}$ viewed as a
quasiparabolic set, and the restriction $\le_{W^+}$ of Bruhat order from
$W^+$.  It follows easily from the subword characterization of Bruhat order
that
\[
x\le_{\cal I} y \implies x\le_{W^+} y;
\]
indeed, any reduced word for $y$ in ${\cal I}$ can be extended to a reduced
word for $y$ in $W^+$ by choosing a reduced word for the minimal perfect
involution in its orbit.  The converse appears to be true by experiment,
but we have only been able to prove the following special case.

\begin{prop}
Suppose the perfect involution $\iota_0\in W^+$ is a diagram automorphism.
Then for $x,y$ in the conjugacy class of $\iota_0$,
\[
x\le_{\cal I} y \iff x\le_{W^+} y.
\]
\end{prop}

\begin{proof}
  Suppose $x\le_{W^+} y$.  If $y=\iota_0$ (note that $\ell(\iota_0)=0$, so
  $\iota_0$ is minimal), the claim is immediate.  Otherwise, there exists a
  simple reflection $s$ such that $s\cdot y<_{\cal I}y$, and in particular
  $sys<_{W^+}sy,ys<_{W^+}y$.  If $x<_{W^+}sx,xs$, then by Lemma
  \ref{lem:bruhat} (applied to $W^+$ as a $W\times W$-set), $x<_{W^+}y$
  implies $x<_{W^+}sy$ implies $x<_{W^+}sys$.  Thus by induction on
  $\hgt(x)+\hgt(y)$, $x<_{\cal I}sys<_{\cal I} y$.  If
  $x>_{W^+}sx,xs>_{W^+}sxs$, then we similarly have $sxs<_{W^+}sys$, so
  that by the same induction, $s\cdot x<_{\cal I} s\cdot y$, and again
  Lemma \ref{lem:bruhat} applied to ${\cal I}$ gives $x<_{\cal I} y$.

  It remains to consider the case $x=sxs>_{W^+}sx,xs$.  But this implies
  that the simple root $\alpha_s$ is taken to its negative by $x$.
  Conjugating $x$ to $\iota_0$ would then give a root $\alpha$ negated by
  $\iota_0$.  However, $\iota_0$ preserves the set of positive roots, so no
  such $\alpha$ can exist.
\end{proof}

\begin{rem}
  Clearly, the first part of the proof holds in general, and thus one
  reduces to showing that if $x=sxs>_{W^+}sx,xs$, then the subset
  $[x,\infty)_{W^+}\cap {\cal I}$ is preserved by the action of $s$ by
  conjugation. (Note that the analogous statement for $[x,\infty)_{\cal I}$
  follows by Corollary \ref{cor:easy_balanced}.)  Using this criterion, one
  may verify the equivalence of these partial orders for various finite
  Coxeter groups, including all sporadic groups, $B_n$, and $D_n$ for $n\le 8$. 
It also follows for $A_{2n-1}$, as there are in that case two nontrivial
  conjugacy classes of perfect involutions: in one, the minimal element is
  a diagram automorphism $\iota_0$ such that $\iota_0x\iota_0 = w_0 x w_0$
  where $w_0$ is the longest element.  Multiplying by the central element
  $\iota_0w_0$ swaps the two conjugacy classes and reverses both Bruhat
  orders.
\end{rem}

\section{The topology of Bruhat order}\label{sec:top}

One of the more important invariants of a poset is its M\"obius function, or
equivalently the Euler characteristics of the order complexes associated to
intervals in the poset.  In the case of the regular representation of $W$
(or equivalently, $W$ viewed as a quasiparabolic $W\times W$-set, since the
orders are the same), there is a classical formula for the M\"obius
function (\cite{VermaD-N:1971,DeodharVV:1977}), which extends to a formula
for the M\"obius function of $W/W_I$.  The latter is somewhat more
complicated, which as we will see is a consequence of the fact $W/W_I$ is
not even.  Indeed, for any bounded even quasiparabolic $W$-set, we have a
very simple formula for the M\"obius function.

\begin{thm}\label{thm:moeb_even}
Let $(X, \hgt)$ be an even  bounded quasiparabolic $W$-set.  Then
the M\"obius function of $(X,\le)$ is given by
\[
\mu(x,y) = \begin{cases} (-1)^{\hgt(y)-\hgt(x)} & x\le y\\
 0 & \text{otherwise.}
\end{cases}
\]
Equivalently, any interval $[x,y]$ with $x<y$ is balanced in the sense that
it has equal numbers of elements with odd and even heights.
\end{thm}

\begin{proof}
The equivalence of the two statements follows immediately from the
definition of the M\"obius function: an interval is balanced iff
\[
\sum_{x\le z\le y} (-1)^{\hgt(y)-\hgt(z)} = 0,
\]
so if every nontrivial interval is balanced, the claimed formula for $\mu$
indeed gives an inverse in the incidence algebra of $(X,\le)$ to the ``zeta
function'' of $(X,\le)$ (i.e., the function which is $1$ precisely when
$x\le y$ and $0$ otherwise).

We now follow the proof in \cite{DeodharVV:1977}, with some simplification
arising from the fact that $X$ is even.  Note first that by Corollary
\ref{cor:easy_balanced}, it follows that any interval $[x,y]$ for which
there exists a simple reflection $s$ with $x<sx$ and $sy<y$ is preserved by
$s$ and thus balanced.  We will use such intervals as a base case for an
induction on $\hgt(x)+\hgt(y)$.

Since $y>x$, $y$ is not minimal, and thus there exists some simple
reflection $s$ such that $sy<y$.  If we had $sx>x$, the interval $[x,y]$
would be balanced per the previous paragraph, so we may as well assume
$sx<x$.  We then claim that one has the following identity of sets:
\[
[sx,y]\setminus [x,y] = [sx,sy]\setminus [x,sy].
\]
Note that each interval other than $[x,y]$ that appears in this expression
is balanced by induction; since clearly $[x,y]\subset [sx,y]$ and
$[sx,sy]\subset [x,sy]$, it will follow that $[x,y]$ is balanced.

In other words, we wish to show that
\[
\{z:z\in X| sx\le z\le y; x\not\le z\}
=
\{z:z\in X| sx\le z\le sy; x\not\le z\}.
\]
The set on the right is clearly contained (since $sy<y$) in the set on the
left, so it suffices to prove the opposite inclusion.

Thus suppose $sx\le z\le y$ but $x\not\le z$; we need to show $z\le sy$.
Since $sx\le z$ but $x\not\le z$, Lemma \ref{lem:bruhat} shows that $x\le
sz$ and thus $sz>z$ (since $z>sz$ would lead to a contradiction).  Since
$z\le y$ and $z<sz$, the same lemma shows $z\le sy$ as required.
\end{proof}

We can then compute the M\"obius function in general by using the M\"obius
function on $\evenX$.

\begin{cor}\label{cor:moeb_odd}
Let $(X, \hgt)$ be a bounded quasiparabolic $W$-set with odd stabilizers.  Then the
M\"obius function of $(X,\le)$ can be described as follows.  If $x\not\le
y$, then $\mu(x,y)=0$.  Otherwise, $\mu(x,y)=(-1)^{\hgt(y)-\hgt(x)}$,
unless there exists a chain
\[
x<r_1x<r_2r_1x<\cdots<y,
\]
$r_i \in R(W)$,
in which the height increases by an even amount at some step; in that case,
$\mu(x,y)=0$.
\end{cor}

\begin{proof}
  The map $x\mapsto \min\{x,s_0x\}$ is a (dual) closure operation on
  $\even{X}$, with quotient poset $X$; we may thus apply
  \cite[Thm. 1]{CrapoHH:1968} (stated only for lattices in the reference,
  but the proof works for any locally finite poset) to compute the M\"obius
  function of $X$ from the M\"obius function of $\evenX$.  To be precise,
\[
\mu(x,y) = \mu(\evenx,\eveny)+\mu(s_0\evenx,\eveny),
\]
since the preimage of $x$ in $\evenX$ is $\{\evenx,s_0\evenx\}$.  Since
$\evenx\le\eveny$, the first term is given by
\[
\mu(\evenx,\eveny)=
(-1)^{\even\hgt(\eveny)-\even\hgt(\evenx)}=(-1)^{\hgt(y)-\hgt(x)}.
\]
The second term is either 0 or the negative of this, and is nonzero
precisely when $s_0\evenx\le \eveny$.  The corollary then follows by Lemma
\ref{lem:hard_even}.
\end{proof}

\begin{rem}
More generally, if one chooses $I\subset S$ such that each $W_I$-orbit is
bounded below, then  by Proposition \ref{prop:proj_para_bruhat}
the operation $\pi_I$
is again a dual closure operation with quotient the subposet $X_I$ of
$W_I$-minimal elements, and one has
\[
\mu_{X_I}(\pi_I{x},\pi_I{y}) = \sum_{z\in W_Ix} \mu_X(z,\pi_I{y}).
\]
It is unclear whether there is a simpler expression for the right-hand
side.
\end{rem}

We also have a generalization of the results of Bj\"orner and Wachs
\cite{BjornerA/WachsM:1982} regarding the shellability (and thus
Cohen-Macaulay-ness) of Bruhat order on Coxeter groups.

\begin{thm}
  Let $(X, \hgt)$ be a bounded quasiparabolic transitive $W$-set.  Then for any
  pair $x\le y$ in $X$, the interval $[x,y]$ is lexicographically
  shellable.  In particular, if $\hgt(y)-\hgt(x)=d+2\ge 2$, then the
  corresponding order complex is homeomorphic to either a $d$-sphere or a
  $d$-cell, depending on whether the Euler characteristic (i.e.,
  $\mu(x,y)$) is $\pm 1$ or $0$.
\end{thm}

\begin{proof}
As in \cite{BjornerA/WachsM:1982}, we fix a reduced word
\[
y = s_1s_2\cdots s_k x_0
\]
where $x_0$ is the minimal element of $X$, and label a saturated chain
\[
x=x_0<x_1<x_2<\cdots<x_l=y
\]
by the sequence $(i_l,i_{l-1},\cdots,i_1)$ where for each $1\le m\le l$,
$i_m$ records the location in the reduced expression for $y$ of the simple
reflection being removed when passing from $x_m$ to $x_{m-1}$.  We then
need only show that the chain with lexicographically minimal label is the
unique chain with increasing label.  The reduction in
\cite{BjornerA/WachsM:1982} to the existence of increasing chains of length
2 carries over mutatis mutandum, but the proof there for said existence
involves cancellation, and thus applies only to Bruhat order in groups.

We thus need to prove the following.  Suppose $x<y$ with
$\hgt(y)=\hgt(x)+2$, and let $y=s_1\cdots s_kx_0$ as above.  Then there
exists (uniqueness follows as in \cite{BjornerA/WachsM:1982}) a pair $i<j$
such that
\[
x = s_1\cdots\widehat{s_i}\cdots\widehat{s_j}\cdots s_k x_0
<
s_1\cdots\widehat{s_i}\cdots s_k x_0
<
y.
\]
Let $i$ be the minimum index such that
\[
x<s_1\cdots\widehat{s_i}\cdots s_k x_0<y.
\]
Then $x$ is obtained by removing some other simple reflection $s_j$ from
this word.  We claim that $j>i$, and prove this by induction on $\hgt(y)$.

We first observe that $j$ cannot be $1$ (note that this addresses the base
case $\hgt(y)=2$). Indeed, we then see that
\[
x=s_2\cdots\widehat{s_i}\cdots s_kx_0<s_2\cdots s_kx_0<s_1\cdots s_kx_0=y
\]
since the fact that the chosen expression for $y$ was reduced implies that
\[
\hgt(s_2\cdots s_kx_0)=\hgt(y)-1.
\]
In other words, $i$ was not minimal, as we could have take $i=1$.

Now, apply $s_1$ to the chain
\[
x=s_1\cdots\widehat{s_j}\cdots\widehat{s_i}\cdots s_k x_0
<
s_1\cdots\widehat{s_i}\cdots s_k x_0
<
s_1\cdots s_kx_0=y
\]
Each of these expressions is reduced, so removing the $s_1$'s preserves the
ordering, and in particular $s_1x<s_1y$ with $\hgt(s_1y)=\hgt(s_1x)+2$.
Thus by induction, if $i'$ is the corresponding minimal choice for passing
from $s_1y$ to $s_1x$ (retaining the labels from $y$, since only the
relative ordering matters), then there is a unique $j'>i'$ such that we
have the chain
\[
\label{eq:chain}
s_1x = s_2\cdots\widehat{s_{i'}}\cdots\widehat{s_{j'}}\cdots s_k x_0
<
s_2\cdots\widehat{s_{i'}}\cdots s_k x_0
<
s_2\cdots s_k x_0=s_1y.
\]
Since $j'$ is unique, we cannot have $i'=i$, since that would force
$j'=j<i=i'$. 
  On the other hand $i$ would be a legitimate choice for $i'$, and
thus minimality implies $i'<i$.  Thus by the original minimality of $i$,
multiplying {\em this} chain by $s_1$ cannot preserve the ordering.  Again
the fact that our original expression for $y$ was reduced implies that
$\hgt(y)>\hgt(s_1y)$, so the only way to break this chain is to change the
ordering of the first two elements.  But this implies by Lemma
\ref{lem:bruhat} that $s_1$ must {\em swap} the first two elements.  In
other words, the middle element in the chain \eqref{eq:chain} gives an expression for $x$,
and we may thus construct the alternate chain
\[
x=s_2\cdots\widehat{s_{i'}}\cdots s_k x_0
<
s_2\cdots s_k x_0
<
s_1\cdots s_k x_0=y.
\]
But then we could again have chosen $i=1$ in the first place.

The remaining topological statements follow from the above together with
the fact that by our previous calculation of the M\"obius function (i.e.,
Euler characteristic), any interval of length 2 has either one or two
intermediate elements.  In other words, the complex is thin.
\end{proof}
\section{Hecke algebra modules}\label{sec:hecke}

Recall that for a Coxeter system $(W,S)$, the Hecke algebra $H_W(q)$ is the
$\Z[q]$-algebra with generators $T(s)$ for $s\in S$ with relations
\[
(T(s)T(t))^{m(s,t)/2}
=
(T(t)T(s))^{m(s,t)/2},
\]
if $m(s,t)$ is even,
\[
(T(s)T(t))^{(m(s,t)-1)/2} T(s)
=
(T(t)T(s))^{(m(s,t)-1)/2} T(t),
\]
if $m(s,t)$ is odd, and
\[
(T(s)-q)(T(s)+1)=0.
\]
More generally, one can choose parameters $q(s)$ for each $s$ with the
proviso that $q(s)=q(t)$ whenever $s$ and $t$ are conjugate.  However, we
will only consider the case of equal parameters in the sequel.  The Hecke
algebra has a natural basis $T(w)$ for $w\in W$ given by taking a reduced
expression $w=s_1s_2\cdots s_k$, $k=\ell(w)$, and defining
\[
T(w) = T(s_1)T(s_2)\cdots T(s_k);
\]
since the Hecke algebra satisfies the braid relations, this is
well-defined.  Moreover, one finds that in this basis, the generators have
a particularly simple action:
\[
T(s)T(w) = \begin{cases}
T(sw) & \ell(sw)>\ell(w)\\
(q-1)T(w)+qT(sw) & \ell(sw)<\ell(w),
\end{cases}
\label{eq:hecke_left}
\]
and similarly
\[
T(w)T(s) = \begin{cases}
T(ws) & \ell(ws)>\ell(w)\\
(q-1)T(w)+qT(ws) & \ell(ws)<\ell(w).
\end{cases}
\label{eq:hecke_right}
\]
The Hecke algebra $H_W(q)$ has two $1$-dimensional representations: the trivial
$\1_+$
and sign (or alternating) $\1_-$ on which each generator $T(s)$ acts
as $\epsilon_+:=q$, respectively as $\epsilon_-:=-1$. 
We will also refer to the restriction to any subalgebra as its trivial, respectively
sign, representation.

When we specialize to $q=1,$ 
$H_W(1)\cong\Z[W]$,
and thus this is indeed a deformation of the group algebra of $W$.  More
generally, if $W$ is finite, then
$H_W(q)\otimes_{\Z[q]} \C \cong \C[W]$  
 for generic $q \in \C$, as
in that case both algebras are semisimple (since any deformation of a
semisimple algebra is generically semisimple); however, unlike the
isomorphism for $q=1$, this isomorphism does not respect the natural basis.

Any Coxeter homomorphism $\phi:W\to W'$ induces two natural homomorphisms
$\phi_{\pm}$ of the corresponding Hecke algebras, by taking
\[
\phi_{\pm}(T(s)) = 
\begin{cases}
T(\phi(s)) & \phi(s)\in S'\\
\epsilon_{\pm} & \phi(s)=1.
\end{cases}
\]
 Of course if $\phi(S)=S'$, then
$\phi_+=\phi_-$, and we may omit the subscript.  It suffices to check that
the braid relation of length $m' = m'(\phi(s), \phi(t))$ is satisfied,
as this implies the braid
relation of length $m(s,t)$.  Indeed, if one splits the left-hand side of
the the braid relation of length $m(s,t)$ into subwords of length $m'$, and
applies the braid relation of length $m'$ to each such subword, one obtains
the right-hand side of the desired braid relation.

As mentioned in the introduction, our original motivation for introducing
quasiparabolic subgroups was to construct modules for the Hecke algebra of
$W$ naturally deforming permutation representations.  

From one perspective, the deformation problem is trivial (at least for
finite $W$), as for generic $q \in \C$, $H_W(q) \otimes_{\Z[q]} \C$ is semisimple,
and thus each
irreducible representation of $\C[W]$ deforms to an irreducible
representation of $H_W(q) \otimes_{\Z[q]} \C$. 
It follows that any representation deforms:
simply decompose it as a sum of irreducibles, and deform the irreducibles
independently.  This deformation suffers from two significant drawbacks,
however.  Not only are the formulas for the action of the generators on
such a representation quite complicated, but also the coefficients of the
action tend to have denominators, in sharp contrast to the Hecke algebra
itself.

Given a standard parabolic subgroup $W_I$ of $W$, there is a natural
parabolic subalgebra $H_I(q)$ of $H_W(q)$ generated by those $T(s)$, $s \in
I$.  In that case, it also makes sense to induce representations of
$H_I(q)$ to $H_W(q)$, via the functor
$\Ind_{H_{I}(q)}^{H_W(q)} \text{--} := H_W(q) \otimes_{H_I(q)} \text{--}$.
The analogue of the permutation representation of $W$ on $W/W_I$ is the
representation induced from the trivial representation $\1_+$ of $H_I(q)$.
(Theorem \ref{thm:induced_hecke} below gives a quasiparabolic interpretation of slightly more general induced modules.)
This representation has a basis naturally indexed by cosets $W/W_I$, and
the action of the generators on this basis has a combinatorial description
analogous to that given by \eqref{eq:hecke_left}, except that there is an
additional case in which the height does not change.

In general, not all subgroups of $W$ deform to a subalgebra of $H_W(q)$, so
this induction construction cannot be used to deform the corresponding
permutation representation.  The parabolic case suggests a natural
combinatorial action of the generators $T(s)$ in the case of a scaled
$W$-set, but there are consistency conditions that must be satisfied in
order for this to give a homomorphism.  When the scaled $W$-set is
quasiparabolic, these conditions are indeed satisfied, and we have the
following.

\begin{thm}
Let $(X,\hgt)$ be a quasiparabolic $W$-set, and let $T(X)$ be the free
$\Z[q]$-module generated by elements $T(x)$ for $x\in X$.  For $s\in S$,
define endomorphisms $T_{\pm}(s)$ of $T(X)$ by
\[
T_\pm(s)T(x) = \begin{cases}
T(sx) & \hgt(sx)>\hgt(x)\\
\epsilon_{\pm} T(x)  & \hgt(sx)=\hgt(x)\ (\implies sx=x)\\
(q-1) T(x)+qT(sx) & \hgt(sx)<\hgt(x)
\end{cases}
\]
Then the map $T(s)\mapsto T_+(s)$ (respectively $T(s)\mapsto T_-(s)$)
afford $T(X)$ the structure of a $H_W(q)$-module, denoted by
$H^{\pm}_X(q)$.
\end{thm}

\begin{rem}
If $X$ has even stabilizers, then $H^+_X(q)$ and $H^-_X(q)$ are naturally
isomorphic, so in this case we may omit the sign.
\end{rem}

\begin{proof}
We need simply prove that $T_\pm(s)$ satisfy the relations of the Hecke
algebra.  The quadratic relation is straightforward (simply check the three
cases), so it remains to check the braid relations.  Since a braid relation
only involves two simple reflections (say $s$ and $t$), we may restrict $X$
to the corresponding rank 2 parabolic subgroup of $W$, or equivalently, may
assume that $W$ has rank 2.  We may also assume that $X$ is transitive, as
$H^\pm_X(q)$ is clearly a direct sum over orbits.

If $|W|=\infty$, there is nothing to check (as there is no braid relation
in that case).  Otherwise $X$ has a minimal element, and by Lemma
\ref{lem:odd_qp_has_simple}, the stabilizer of that element has the form
\[
\langle (st)^{m'}\rangle,\ 
\langle (st)^{m'},s\rangle,\ 
\langle (st)^{m'},t\rangle
\]
for some $m'$ dividing $m(s,t)$.  We may thus (replacing $W$ by a quotient
as necessary and using the fact that Coxeter homomorphisms extend to
Hecke algebras) reduce to the case that the stabilizer of the minimal
element is trivial or generated by a simple reflection.

If the stabilizer is trivial, this is simply the regular representation of
$H_W(q)$, so the result is immediate.  Otherwise, say the stabilizer is
$\langle s\rangle$.  Then there is a natural isomorphism
\[
T(X)\cong H_W(q) (T(s)-\epsilon_{\mp})
\cong
\Ind^{H_W(q)}_{H_{\langle s\rangle}(q)} \1_\pm
,
\]
taking
\[
T(wx)\mapsto T(w) (T(s)-\epsilon_{\mp})
\]
whenever $\hgt(wx)=\ell(w)+\hgt(x)$, and this respects the action of
$T_{\pm}(s)$.
\end{proof}

\begin{rems}
Note that even if $s$ and $t$ are not conjugate, there can be orbits with
$m'$ odd, forcing the corresponding parameters $q(s)$, $q(t)$ to agree;
this is the reason for our simplifying assumption $q(s)\equiv q$.
\end{rems}

\begin{rems}
In fact, in order for this to extend to a homomorphism, it suffices that
the restriction of $X$ to any rank 2 subgroup of $W$ be quasiparabolic.
Experimentally, this weaker condition is very nearly the same as
quasiparabolicity (see Example \ref{eg:bozo} below), but fails to be
preserved by many of the constructions above, especially Theorem
\ref{thm:quotient} and its corollaries.
\end{rems}

Taking $q=0$ gives the following combinatorial fact.  Define the Hecke
monoid of a Coxeter group $W$ to be the quotient $M(W)$ of the free monoid
on idempotent generators $M(s)$, $s\in S$ by the braid relations.

\begin{cor}
There is a natural action of the Hecke monoid $M(W)$ on $X$ in which the
generators act by
\[
M(s)x = \max\{x,sx\}.
\]
\end{cor}

\begin{proof}
We observe that apart from some sign changes, this is precisely the action
of the generators of $H_W(0)$ on $H^-_X(0)$.
\end{proof}

\begin{rem}
This action is compatible with Bruhat order in the sense of Richardson and
Springer \cite{RichardsonRW/SpringerTA:1990}.
\end{rem}

The construction in the case of a dihedral group with odd stabilizers
generalizes as follows.

\begin{prop}\label{prop:even_module}
Let $\evenX$ be the even double cover of $X$.  Then
\[
H^\pm_X(q) \cong H_{\evenX}(q) (T(s_0)-\epsilon_{\mp}),
\]
where $s_0$ generates the factor $A_1$ acting on $\evenX$.  Over
$\Z[q,1/(q+1)]$, there is a natural isomorphism of $H_q(W)$-modules
\[
H_{\evenX}(q) \cong H^+_X(q)\oplus H^-_X(q).
\]
\end{prop}

Specializing to $q=1$, we find that $H^+_X(1)$ is simply the permutation
representation associated to $X$, while $H^-_X(1)$ is the tensor product of
this with the sign representation of $W$.

Recall the following notation.
If $M$ is an $H_W(q)$-module and $N$ is an $H_{W'}(q)$-module with both
$M$ and $N$ free over $\Z[q]$, then their tensor product, which carries the
obvious $H_W(q) \otimes H_{W'}(q)$-module structure,  is denoted $M \boxtimes N$.
Note furthermore that $H_W(q) \otimes_{\Z[q]} H_{W'}(q) \cong H_{W \times W'} (q)$.

\begin{prop}
If $(X, \hgt)$ and $(X', \hgt')$ are quasiparabolic $W$- and $W'-$sets respectively, then
there are natural isomorphisms
\begin{align}
H^+_X(q)\boxtimes H^+_{X'}(q)&\cong H^+_{X\times X'}(q)\\
H^-_X(q)\boxtimes H^-_{X'}(q)&\cong H^-_{X\times X'}(q).
\end{align}
\end{prop}

\begin{prop}
Let $(X,\hgt)=(W,\ell)$, viewed as a $W\times W$-set.  Then there is a
natural $H_W(q)\otimes H_W(q)$-module isomorphism
\[
H_X(q)\cong H_W(q).
\]
\end{prop}

\begin{rem}
Here we identify $H_W(q)^{\text{op}}$ with $H_W(q)$ via the involution
$T(w)\mapsto T(w^{-1})$.
\end{rem}

It is not clear how to interpret the construction of Theorem
\ref{thm:quotient} in the Hecke algebra setting.  However, the special case
of a quasiparabolic subgroup of a parabolic subgroup is straightforward.

\begin{thm}\label{thm:induced_hecke}
Let $H\subset W_I$ be a quasiparabolic subgroup of the parabolic subgroup
$W_I\subset W$.  Then there is a natural isomorphism
\[
\Ind_{H_{I}(q)}^{H_W(q)} H_{W_I/H}(q)\cong H_{W/H}(q).
\]
\end{thm}

\begin{proof}
  It follows from standard results on induction of Hecke algebra modules
  from parabolic subalgebras that the induced module has a basis over
  $\Z[q]$ of the form $T(u)\otimes T(vH)$ with $(u,vH)$ ranging over
  $W^I\times W_I/H$.  The desired isomorphism is then given by
\[
T(u)\otimes T(vH) \mapsto T(uvH).
\]
This is clearly an isomorphism of free $\Z[q]$-modules, so it remains only
to show that it is a homomorphism of $H_W(q)$-modules.  Observe that there
is in general a homomorphism
\[
T(w)\otimes T(vH)\mapsto T(w)T(vH)
\]
obtained via the left adjointness of induction to restriction from the
natural inclusion $H_{W_I/H}(q)\subset H_{W/H}(q)$ of $H_{I}(q)$-modules.
When $w=u\in W^I$, this map takes
\[
T(u)\otimes T(vH)\mapsto T(u)T(vH)=T(uvH)
\]
as required.  (For the second equality (i.e., the fact that
$\hgt(uvH)=\ell(u)=\hgt(vH)$), see Remark 1 following Corollary
\ref{cor:induce_from_para}.)
\end{proof}

One nice property of our Hecke algebras modules is the existence of a
natural symmetric bilinear 
form.  Define a pairing on $T(X)$ by
\[
\langle T(x),T(y)\rangle = \delta_{xy} q^{\hgt(x)}
\]
and extending linearly.

\begin{prop}
The linear transformations $T_\pm(s)$ are self-adjoint with respect to this
pairing.
\end{prop}

\begin{proof}
We simply need to verify that
\[
\langle T_{\pm}(s)T(x),T(y)\rangle
=
\langle T(x),T_\pm(s)T(y)\rangle.
\]
If $\{x,sx\}\ne \{y,sy\}$, then both sides are 0,
and if $y=x$ the claim is obvious by symmetry.  We thus reduce to the case
$y=sx\ne x$.  Moreover, we may assume $\hgt(sx)=\hgt(x)+1$, as we can
otherwise exchange $x$ and $y$.  But then
\[
\langle T(s)T(x),T(sx)\rangle
=
q^{\hgt(x)+1}
\]
while
\[
\langle T(x),T(s)T(sx)\rangle
=
\langle T(x),q T(x)+(q-1)T(sx)\rangle
=
q q^{\hgt(x)}.
\]
\end{proof}

In particular, in the case $X=W$, this is just the usual invariant inner
product on the Hecke algebra, namely $\langle T(w),T(w')\rangle =
\delta_{ww'} q^{\ell(w)}$.

\begin{thm}
Let $W$ be a finite Coxeter group, and let $(X, \hgt)$ be a transitive
quasiparabolic $W$-set with minimal element $x_0$.  Then over $\Z[q,q^{-1}]$,
there is a natural injection
\[
H^{\pm}_X(q)\to H_W(q)
\]
of Hecke algebra modules.
\end{thm}

\begin{proof}
Simply take the adjoint of the surjection
\[
H_W(q)\to H^{\pm}_X(q)
\]
given by
\[
T(w)\mapsto T(w)T(x_0).
\]
\end{proof}

In the case $X$ is transitive and bounded from below, the corresponding
Hecke algebra modules are cyclic, generated by $T(x_0)$, where $x_0$ is the
minimal element of $X$.  One can thus express $H^{\pm}_X(q)$ as a quotient
$H_W(q)/I_X$, for some left ideal $I_X$, the annihilator of $T(x_0)$. 
It is therefore natural to ask
whether we can give a nice set of generators for $I_X$.  By Proposition
\ref{prop:even_module}, one can essentially reduce to the case $X$ even;
the ideal $I_X$ can be obtained from the ideal $I_{\evenX}\subset H_W(q)$
by adding one generator of the form $T(s)-\epsilon_{\mp}$, $s \in S$.
\begin{lem}
Suppose $H$ is an even quasiparabolic subgroup of $W$.  The ideal $I_{W/H}$
is generated by elements of the form $T(w)-T(w')$ where $wH=w'H$ and
$\hgt(wH)=\ell(w)=\ell(w')$.
\end{lem}

\begin{proof}
Since for any minimal coset representative $w$, $T(w)T(H)=T(wH)$, we see
that the above elements are indeed in the ideal, so it remains to show that
they generate.  Choose a map $\phi:W/H\to W$ with the property that
$\phi(x)H=x$ and $\hgt(x)=\ell(\phi(x))$ for all $x\in W/H$; i.e., $\phi$
chooses a minimal representative of each coset.  The elements
$T(\phi(x))T(H)=T(x)$, $x\in W/H$ form a basis of $T(W/H)$, so we need
simply to show that modulo the possibly smaller ideal, every element $T(w)$
is congruent to a linear combination of elements $T(\phi(x))$.

By induction on $\ell(w)$, it suffices to show this for an element of the
form $T(s)T(\phi(x))$.  If $\hgt(sx)=\hgt(x)+1$, then $s\phi(x)$ is a
minimal coset representative, and therefore
\[
T(s\phi(x))-T(\phi(sx))
\]
is one of our chosen generators.  Otherwise, $\hgt(sx)=\hgt(x)-1$, in which
case $s\phi(sx)$ is a minimal coset representative,
\[
T(\phi(x))-T(s\phi(sx))
\]
is one of our chosen generators, and
\[
T(s\phi(x)) - (T(s)-(q-1))\bigl(T(\phi(x))-T(s\phi(sx))\bigr)
=
(q-1)T(\phi(x))+q T(\phi(sx)).
\]
\end{proof}

This set of generators is highly redundant, however, so we would like to
find a small subset that still generates $I_X$.  In general, it is too much
to hope for this subset to be finite, but we can still reduce the
complexity considerably.  Note first that if $w$, $w'$ are minimal
representatives of the same coset of $H\subseteq W$, and there exists a
simple reflection $s$ such that $\ell(sw)=\ell(sw')=\ell(w)-1$, then we
have
\[
T(w)-T(w') = T(s)(T(sw)-T(sw')),
\]
so any such generator is redundant.  More generally, if we could find an
element $w''$ such that $\ell(sw)=\ell(sw'')=\ell(w)-1$ and
$\ell(tw')=\ell(tw'')=\ell(w')-1$, $s,t\in S$, then
\[
T(w)-T(w') = \bigl(T(w)-T(w''))-\bigl(T(w')-T(w''))
\]
is also redundant.  In general, $T(w)-T(w')$ will be redundant so long as
there is any path from $w$ to $w'$ along the above lines.  There is,
however, an implied condition on $s$ and $t$, namely that $wH$ is
$\{s,t\}$-maximal, and moreover that the $\langle s,t\rangle$-orbit of $wH$
has size $2m(s,t)=|\langle s,t\rangle|$.  After all, $w''$ must have a
reduced expression beginning with one side of the braid relation between
$s$ and $t$, and removing that subword gives a reduced expression for the
minimal element of the $\langle s,t\rangle$-orbit.  This suggests looking
at the following graph $\Gamma_x$ for each element $x\in W/H$: the vertices
of $\Gamma_x$ are precisely the simple reflections such that
$\hgt(sx)=\hgt(x)-1$, while the edges are the pairs $\{s,t\}$ such that the
$\langle s,t\rangle$-orbit of $x$ has size $2m(s,t)$.  The above
considerations tell us that if the graph is connected, then {\em all}
relations of the form $T(w)-T(w')$ arising from $x$ are redundant, i.e.,
can be expressed in terms of relations arising from elements of smaller
height.  If the graph is not connected, we need simply add enough relations
so that the corresponding additional edges make the graph connected.

Given a pair $\{s,t\}$ of vertices of $\Gamma_x$ which is {\em not} an
edge, there is a particularly nice choice of generator, as follows.
Consider the $\langle s,t\rangle$-orbit of $x$, of size $2k$, $k$ strictly
dividing $m(s,t)$, and let $y$ be the minimal element of that orbit, of
height $\hgt(y)=\hgt(x)-k$.  Choose an expression $y=wH$ with
$\ell(w)=\hgt(y)$, and observe that
\[
(st)^{k/2} w\quad\text{and}\quad (ts)^{k/2} w
\]
both give rise to reduced expressions for $x$ (with the obvious
interpretation of the $k/2$ power when $k$ is odd, as in the braid
relation), and thus
\[
T((st)^{k/2} w)-T((ts)^{k/2} w)
\]
is in the annihilator ideal.  We call such generators {\em dihedral} generators based
at $x$.  Since adding an edge connecting vertices in different components
of $\Gamma_x$ reduces the number of components by 1, we obtain the
following result.

\begin{thm}
The ideal $I_X$ has a generating set consisting of dihedral generators, in
such a way that the number of dihedral generators based at $x$ in the
generating set is one less than the number of components of $\Gamma_x$.
\end{thm}

\begin{rem}
  When $X$ is finite, this set of generators is finite (even if $W$ itself
  is not), and in practice is quite small.  For instance, in Example
  \ref{eg:perfect} below, we mention a quasiparabolic $E_8\times A_2\times
  A_2$-set $X$ of size 113400 for which the above generating set consists
  of only 8 elements, whereas the first generating set we gave is {\em
    significantly} larger.  Also observe that this gives rise to a fairly
  small set of generators of the quasiparabolic subgroup $H$ stabilizing
  the minimal element.
\end{rem}

\section{Poincar\'e series}\label{sec:PS}

Given a quasiparabolic subgroup $H\subset W$, one natural invariant is the
Poincar\'e series
\[
\PS_{W/H}(q):=\sum_{x\in W/H} q^{\hgt(x)} \in \Z[[q]].
\]
In the case $H=W_I$, one has
\[
\PS_{W/W_I}(q)\PS_{W_I}(q) = \PS_W(q),
\]
where the other two Poincar\'e series are in terms of the regular
representation (i.e., the usual Poincar\'e series of a Coxeter group).  In
particular, it follows that when $W$ is finite, the zeroes of
$\PS_{W/W_I}(q)$ are roots of unity, as this is true for $\PS_W(q)$.  The
proof of this fact relies heavily on the fact that the restriction of $W$
to a $W_I$-set is a union of $[W:W_I]$ copies of $W_I$, with appropriately
shifted heights.  No such decomposition exists for a general quasiparabolic
subgroup, and yet the following still holds.

\begin{thm}\label{thm:Poincare_divides}
Let $H$ be a quasiparabolic subgroup of the finite Coxeter group $W$.  Then
$\PS_{W/H}(q)$ is a divisor of $\PS_W(q)$.
\end{thm}

\begin{proof}
We compute the product
\[
\label{eq:sumsum}
(\sum_{w\in W} T_+(w))(\sum_{x\in W/H} T(x))
\]
in two different ways.  On the one hand, for any $w\in W$,
\[
T_+(w)\sum_{x\in W/H}T(x) = q^{\ell(w)}\sum_{x\in W/H} T(x),
\]
by induction on $\ell(w)$, so the product \eqref{eq:sumsum} is
\[
\PS_W(q) \sum_{x\in W/H} T(x).
\]
On the other hand, for all $x\in W/H$,
\[
(\sum_{w\in W} T_+(w))T(x) = q^{\hgt(x)} (\sum_{w\in W} T_+(w))T(H),
\]
by induction on $\hgt(x)$.
Therefore,
\[
\PS_{W/H}(q) (\sum_{w\in W} T_+(w))T(H)
=
\PS_W(q) \sum_{x\in W/H} T(x).
\]
But the coefficients of
\[
(\sum_{w\in W} T_+(w))T(H)
\]
lie in $\Z[q]$, and thus so does the ratio of Poincar\'e series.
\end{proof}

\begin{rem}
We may thus define a Poincar\'e series of $H$ as the ratio
\[
\PS_H(q) = \frac{\PS_W(q)}{\PS_{W/H}(q)}.
\]
As observed in Remark 1 following Corollary \ref{cor:induce_from_para},
this is preserved by induction from parabolic subgroups.  Note that because
the proof that $\PS_H(q)$ is a polynomial is not combinatorial in nature,
we do not obtain an interpretation of $\PS_H(q)$ as a generating function for
elements of $H$.  There does, however, seem to be a surprising
amount of structure in $\PS_H(q)$.  For instance, Example \ref{eg:E} below discusses
a quasiparabolic subgroup $H\subset E_8$ with
\[
\PS_H(q) = \frac{(1-q^2)(1-q^{12})(1-q^{20})(1-q^{30})}{(1-q)^4}
         = \PS_{H_4}(q),
\]
and indeed $H$ is abstractly isomorphic to the Coxeter group $H_4$; similar
agreement appears to occur whenever the quasiparabolic subgroup is
abstractly isomorphic to a Coxeter group.  Even more strikingly, there are
several examples of Poincar\'e series of a similar form in which the
degrees of invariants are replaced by degrees of invariants in
characteristic 2, see Example \ref{eg:char2_PS} below.
\end{rem}

\begin{cor}
With hypotheses as above,
\[
q^m \PS_{W/H}(q^{-1})  = \PS_{W/H}(q),
\]
where $m=\max_{x\in W/H}\hgt(x)$.
\end{cor}

\begin{proof}
It is known that the zeros of $\PS_W(q)$ are all roots of unity, and thus
the same is true for its divisor $\PS_{W/H}(q)$.  Since $\PS_{W/H}(q)$ has
integer coefficients, its roots are permuted by the absolute Galois group
of $\Q$.  In particular, for every root of $\PS_{W/H}(q)$, its complex
conjugate is also a root, or in other words, the reciprocal of each root is
a root.  But then
\[
q^m \PS_{W/H}(q^{-1})
\]
has the same roots with multiplicities as $\PS_{W/H}(q)$ (note that
$m=\deg(\PS_{W/H}(q))$).  Since $W/H$ has unique maximal and minimal
elements, both polynomials are monic, with the same roots, so must agree.
\end{proof}

If $W$ is infinite, it is too much to hope for the ratio
$\PS_W(q)/\PS_{W/H}(q)$ to be a polynomial.  When $W$ is affine, we suspect
that an analogous statement should hold, to wit that the ratio is a
rational function, with zeros and poles only at roots of unity.  The above
methods appear to be completely insufficient for this case, however.

Despite the symmetry of the corollary, the scaled $W$-sets $(W/H,\hgt)$
and $(W/H,m-\hgt)$ are not in general isomorphic.  In particular, in such a
case, we obtain two different deformations of the same permutation
representation.  We conjecture that not only are these deformations isomorphic,
but also that the isomorphism can be chosen to have 
a particularly nice form.
  If $(X,\hgt)$ is a scaled
$W$-set, let $X^-$ denote the scaled $W$-set $(X,-\hgt)$.

\begin{conj}\label{conj:R}
  Let $H\subset W$ be a quasiparabolic subgroup.  Then there is an
  isomorphism (with coefficients in $\Z[q,1/q]$)
\[
H^{\pm}_{W/H}(q)\cong H^{\pm}_{(W/H)^-}(q)
\]
of $H_W(q)$-modules in which $T^{\pm}(H)$ maps to $T^{\pm}(H)$.
\end{conj}

Recall that on the left, the Hecke algebra acts by
\[
T(s)T_+(x)
=
\begin{cases}
T_+(sx) & \hgt(sx)>\hgt(x)\\
\epsilon_{\pm} T_+(x)  & \hgt(sx)=\hgt(x)\\
(q-1) T_+(x)+qT_+(sx) & \hgt(sx)<\hgt(x)
\end{cases}
,
\]
while on the right, the Hecke algebra acts by
\[
T(s)T_-(x)
=
\begin{cases}
T_-(sx) & \hgt(sx)<\hgt(x)\\
\epsilon_{\pm} T_-(x)  & \hgt(sx)=\hgt(x)\\
(q-1) T_-(x)+qT_-(sx) & \hgt(sx)>\hgt(x)
\end{cases}
\]
We may rewrite this in terms of $T'(s) = q-1-T(s)$, $T'_-(x)=(-q)^{\hgt(x)}T_-(x)$:
\[
T'(s)T'_-(x)
=
\begin{cases}
T'_-(sx) & \hgt(sx)>\hgt(x)\\
\epsilon_{\mp} T'_-(x)  & \hgt(sx)=\hgt(x)\\
(q-1)T'_-(x)+qT'_-(sx) & \hgt(sx)<\hgt(x).
\end{cases}
\]
It follows that the desired isomorphism exists iff the annihilator of
$T(H)$ in $T^{\pm}(W/H)$ is taken to the annihilator of $T(H)$ in
$T^{\mp}(W/H)$ by the automorphism $T(s)\mapsto T'(s)$.

Since in finite cases we can find relatively small generating sets for
these annihilators, it is straightforward in most cases to verify the
existence of these isomorphisms.  In particular, this isomorphism exists in
every finite case we have checked.

Note that the coefficients of such an isomorphism would give an analogue
for quasiparabolic $W$-sets of the $R$-polynomials of Kazhdan-Lusztig
theory.  There is a formula for the latter polynomials due to Deodhar
\cite{DeodharVV:1985} expressed in terms of a generating function for
``distinguished'' subexpressions.  It appears that there is no obstacle to
constructing the corresponding generating function in the quasiparabolic
setting, but Deodhar's proof makes essential use of the fact that the
natural family of recurrences for $R$-polynomials are consistent.  In other
words, if we knew that our isomorphism existed, there would almost
certainly be a formula for the relevant coefficients \`a la Deodhar.

In addition, since the theory of Kazhdan-Lusztig polynomials themselves has
analogues for quotients by parabolic subgroups, we expect there to be a
corresponding analogue for quotients by quasiparabolic subgroups.
\section{Examples}\label{sec:egs}

\begin{eg}\label{eg:bozo}
  We begin with an example of a non-quasiparabolic $W$-set.  Let $B_3$ be
  the hyperoctahedral group of signed permutations, with simple reflections
  $(12)$, $(23)$, and $(3)_-$, and consider the subgroup $H\subset B_3$ of
  order 8 generated by $(13)(2)_-$, $(1)_-$, and $(3)_-$, and the
  corresponding self-dual scaled $B_3$-set $B_3/H$.  This, together with
  its even subgroup is the only indecomposable example we know of a
  non-quasiparabolic subgroup such that all restrictions to rank $2$ are
  quasiparabolic.  We also note that this subgroup violates many of the
  conclusions of our theorems above; for instance, it fails to map to a
  quasiparabolic subgroup of $A_2\times A_1$ (not even when restricted to
  rank 2), and does not induce a well-behaved Bruhat order.
\end{eg}

\begin{eg}\label{eg:perfect}
One of the more fruitful constructions is the fact that perfect involutions
form quasiparabolic $W$-sets.  As we mentioned above, it is relatively
straightforward to classify perfect involutions in finite Coxeter groups.
Note that if $W$ is a product, then the relevant diagram automorphism
permutes the factors of $W$, and if the action on factors is not
transitive, then the corresponding sets of perfect involutions are just
products of the sets of perfect involutions from each orbit.  In addition,
if $W=W'\times W'$ and the diagram automorphism swap the factors, then up
to conjugation by a further diagram automorphism, the only perfect
involutions come from the diagonal.  Thus the only interesting cases arise
from diagram automorphisms of simple Coxeter groups.

In type $A$, as we mentioned above, there are two noncentral conjugacy
classes of perfect involutions, one with the trivial diagram automorphism,
and one with the nontrivial diagram automorphism.  Indeed, it is easy to
see that a perfect involution must act as the inner automorphism
corresponding to a fixed-point-free involution.  (If $r$ is the reflection
swapping a fixed point and a non-fixed point of the image of $\iota$ in
$S_n$, then $\iota r$ would have order $3$ or $6$.)  Note that although the
two resulting quasiparabolic sets are dual (i.e., related by negating
heights), the stabilizers of their respective minimal elements are
qualitatively quite different.  For instance, the centralizer of the
diagram automorphism contains only one simple reflection, while the
centralizer of the minimal fixed-point-free involution in $S_{2n}$ contains
$n$ simple reflections.  However, it follows by considering the
corresponding ideals that the two Hecke algebra modules are isomorphic as
in Conjecture \ref{conj:R} and the discussion following it.

In type $B/C$, a noncentral order 2 element of $B_n$ is a perfect
involution iff its image in $S_n$ is the identity or a fixed-point-free
involution.  In the former case, the small subgroup $z_W(\iota)$ contains
the kernel of the natural Coxeter homomorphism $B_n\to A_{n-1}\times A_1$,
and the image is a product of the form $\Alt_j\times \Alt_{n-j}$, embedded
in the natural way.  In the latter case, $z_W(\iota)$ is just the even
subgroup of $Z_W(\iota)$.

In type $D$, again the image in $S_n$ of a perfect involution must be the
identity or a fixed-point-free involution.  The first case includes all
perfect involutions that involve the nontrivial diagram automorphism, and
is analogous to the $B$ case.  The fixed-point-free involutions now come in
two conjugacy classes (swapped by the diagram automorphism), and the
corresponding quasiparabolic sets are dual (i.e., differ by reversing the
heights).  As mentioned above, this is the one case where we do not know
whether 
$z_W(\iota)$ (an index 4 subgroup of the centralizer of the minimal
$\iota$) is quasiparabolic.

It remains to consider the sporadic cases.  Other than $B_2$, no dihedral
group has a noncentral perfect involution (even including diagram
automorphisms), and similarly for $H_3$ and $H_4$.  For $F_4$, the
noncentral perfect involutions (none of which involve the nontrivial
diagram automorphism) form a single conjugacy class generated by the
longest element of the parabolic $B_2\subset F_4$.  In each of $E_6$,
$E_7$, and $E_8$, the longest element of the parabolic $D_4$ is perfect; in
$E_6$ and $E_7$, one also has another conjugacy class of perfect
involutions giving a dual quasiparabolic set.  (For $E_6$, take the
conjugacy class of the nontrivial diagram automorphism; for $E_7$, the
other minimal involution is the product of three commuting roots.)  

Note that for $E_8$, the small group $z_W(\iota)$ has index $36=|A_2\times
A_2|$ in the centralizer, so we obtain an action of $E_8\times A_2\times
A_2$ on a set of order $113400$.  Explicit computation gives the eight dihedral
generators mentioned above: six of length $2$ and two of length $8$.  For
$E_7$, the small group has index $12=|A_2\times A_1|$ giving a
quasiparabolic action of $E_7\times A_2\times A_1$ on a set of size $3780$.
\end{eg}

Given a transitive quasiparabolic $W$-set $X$, sometimes we can extend the
action to a larger Coxeter group $W'$ in which $W$ is standard parabolic,
while retaining quasiparabolicity.  Note that by Theorem
\ref{thm:bruhat_para}, the Bruhat order as a $W'$-set will be the same as
the original Bruhat order.  We can thus produce a relatively short list of
candidates for the actions of simple reflections in $W'$.  We find, at
least when $X$ has even stabilizers that a simple reflection in $W'$ must
be a special matching in the sense of
\cite{BrentiF/CaselliF/MariettiM:2006}.  There is, of course, the
additional requirement that if we adjoin a new simple reflection, that all
resulting new reflections must only swap comparable elements.  Note,
however, that even when $W'$ is infinite, if $X$ is finite, the image of
$W'$ in $\Sym(X)$ is finite, so it is a finite computation to verify
quasiparabolicity.  Moreover, only the condition of comparability remains
to be checked, via Proposition \ref{prop:QP_via_Bruhat}.  Note that as a
special case, if $W$ is a Weyl group, and the reflection in one of its
highest roots induces a special matching, we can always extend to a
quasiparabolic action the corresponding affine Weyl group, as this will not
enlarge the image of $R(W)$ inside $\Sym(X)$.

\begin{figure}[!ht]
\ifx\JPicScale\undefined\def\JPicScale{1}\fi
\unitlength \JPicScale mm
\begin{center}
\begin{picture}(122.5,75)(0,75)
\linethickness{0.7mm}
\multiput(70,140)(0.24,-0.12){83}{\line(1,0){0.24}}
\multiput(90,130)(0.12,-0.24){83}{\line(0,-1){0.24}}
\multiput(90,90)(0.12,0.24){83}{\line(0,1){0.24}}
\multiput(70,80)(0.24,0.12){83}{\line(1,0){0.24}}
\multiput(50,90)(0.24,-0.12){83}{\line(1,0){0.24}}
\multiput(40,110)(0.12,-0.24){83}{\line(0,-1){0.24}}
\multiput(40,110)(0.12,0.24){83}{\line(0,1){0.24}}
\multiput(50,130)(0.24,0.12){83}{\line(1,0){0.24}}
\linethickness{0.7mm}
\put(100,110){\line(1,0){20}}
\linethickness{0.7mm}
\put(20,110){\line(1,0){20}}
\put(122.5,110){\makebox(0,0)[bl]{AE8}}
\put(100,112.5){\makebox(0,0)[bl]{E7}}
\put(92.5,130){\makebox(0,0)[bl]{E6}}
\put(92.5,90){\makebox(0,0)[bl]{E678}}
\put(42.5,90){\makebox(0,0)[bl]{A8}}
\put(20,110){\circle*{2}}
\put(40,110){\circle*{2}}
\put(100,110){\circle*{2}}
\put(120,110){\circle*{2}}
\put(70,140){\circle*{2}}
\put(70,80){\circle*{2}}
\put(50,130){\circle*{2}}
\put(50,90){\circle*{2}}
\put(90,130){\circle*{2}}
\put(90,90){\circle*{2}}
\end{picture}
\end{center}
\caption{Coxeter diagram of $O_{10}$}
\label{fig:ears}
\end{figure}
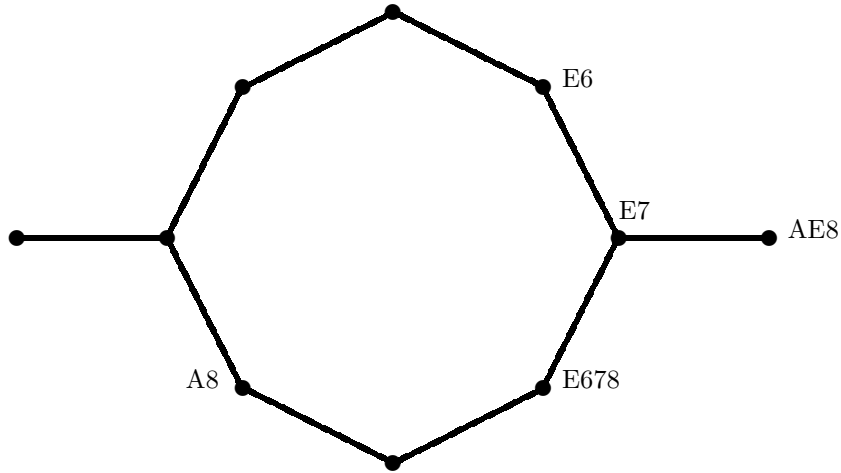

\begin{eg}
  Let $\iota$ be the diagram automorphism of $E_6$, which as we have
  observed is perfect.  We can directly check that the normal subgroup
  $z_{E_6}(\iota)$ (of index $6=|A_2|$ in the centralizer) is
  quasiparabolic, and thus gives rise to a transitive quasiparabolic
  $E_6\times A_2$-set of size $270$.  The stabilizer in $E_6\times A_2$ of
  the minimal element is abstractly isomorphic to $F_4$, and the
  corresponding Poincar\'e series agree.

  There are, it turns out, precisely $10$ special matchings on this set, of
  which $8$ are accounted for by the simple reflections in $E_6\times A_2$.
  If we adjoin the remaining $2$ special matchings, we obtain an action of
  a Coxeter group $O_{10}$ of rank $10$ on $X$, with diagram as in Figure
\ref{fig:ears}

  (The image in $\Sym(X)$ is isomorphic to $E_8/Z(E_8)$,
  but this does not factor through a Coxeter homomorphism $O_{10}\to E_8$.)
  It is computationally straightforward to verify that this gives a
  quasiparabolic action of the quite large Coxeter group $O_{10}$ on $X$;
  moreover, the diagram automorphisms of $O_{10}$ extend to
  Bruhat-preserving automorphisms of $X$.  Since $O_{10}$ is simply-laced,
  it has only one conjugacy class of reflections, which correspond to the
  $120$ reflections of $E_8$.  Various parabolic subgroups of $O_{10}$ act
  transitively: any subgroup containing one of the two parabolic subgroups
  of type $E_6$ is transitive, as are the four parabolic subgroups of type
  $A_8$.

  In particular, we obtain a quasiparabolic subgroup of $E_7$ abstractly
  isomorphic to $\F_2^7\semidirect \PGL_3(2)$, and a quasiparabolic
  subgroup of $E_8$ abstractly isomorphic to $2^{1+6}\semidirect\PGL_4(2)$
  (where $2^{1+6}$ denotes a $2$-group of order $128$ with center of order
  $2$).  We also obtain quasiparabolic actions of the corresponding affine
  groups, in which the additional reflection has the same action on $X$ as
  the reflection in the highest root.  The quasiparabolic subgroup of $A_8$
  we obtain is actually contained in a maximal parabolic subgroup of type
  $A_7$ which we consider in the next example.

$O_{10}$ contains various transitive parabolic subgroups which can be obtained in
several ways, and in particular by removing the $ s \in S$  corresponding to
marked nodes as in Figure \ref{fig:ears}.
One obtains $A_8$ removing the two nodes marked A8 and AE8;
one obtains $E_6 \times A_2$ removing the two nodes marked E6 and E678;
one obtains $\tilde{E_8}$ removing the node marked  E678, and then $E_8$ by
further removing AE8;
one obtains $\tilde{E_7} \times A_1$ removing the node marked  E7,
and then $E_7 \times A_1$ by further removing E678.
\end{eg}

\begin{eg}\label{eg:char2_PS}
The action of $A_8$ in the previous example is induced from an action of
$A_7$ on a set of size $30$, with quasiparabolic subgroup isomorphic to
$\AGL_3(2)$.  The Poincar\'e series of this subgroup has the striking
structure
\[
\PS_{\AGL_3(2)}(q) = \frac{(1-q^4)(1-q^6)(1-q^7)(1-q^8)}{(1-q)^4}.
\]
If $\AGL_3(2)$ had been a Coxeter group, we would have concluded that its
invariant ring was freely generated by elements of degrees $4$, $6$, $7$,
and $8$.  Surprisingly, there is indeed an action of $\AGL_3(2)$ with such
an invariant ring, but in characteristic 2.  Indeed, there are two actions
of $\AGL_3(2)$ on $\bar{\F_2}^4$, one with an invariant subspace of
dimension 1, and one with an invariant subspace of dimension 3.  The latter
has precisely the desired invariants.

We could also have obtained this set as in the previous example, beginning
with either of the classes of perfect involutions in $A_5$.  There is also
a transitive action of $A_4$, induced from the action of $A_3$ on $A_2$ via
the natural Coxeter homomorphism.  In the other direction, the reflection
in the highest root of $A_7$ is a special matching, and thus we obtain a
quasiparabolic action of $\tilde{A}_7$ on this set.

This characteristic $2$ invariant theory phenomenon also arises for two
other quasiparabolic subgroups of symmetric groups.  First, the transitive
action of $A_6$ on the same set of size $30$ has stabilizer $\PGL_3(2)\cong
\GL_3(2)$.  The invariant ring of $\GL_3(2)$ in its $3$-dimensional
characteristic $2$ representation is freely generated by invariants of
degrees $4$, $6$, and $7$, which again agrees with the Poincar\'e series of
the subgroup.

The other example comes from the transitive action of $\Alt_5$ on a set of
$6$ elements.  It turns out that one representative of the resulting
conjugacy class of subgroups of $A_5\cong S_6$ is actually quasiparabolic,
of index $12$.  There is a $3$-dimensional representation of $\Alt_5$ in
characteristic $2$, namely $\Alt_5\cong O^-_3(\F_4)$ (in the version with a
$2$-dimensional invariant subspace), with invariant ring freely generated
by elements of degrees $2$, $5$, and $6$.  Once more, the Poincar\'e series
of this quasiparabolic $\Alt_5$ agrees with the product suggested by the
degrees of invariants.

It should be noted, however, that not all quasiparabolic subgroups have a
Poincar\'e series of this form.  For instance, the Poincar\'e series of the
index $113400$ quasiparabolic subgroup $H\subset E_8$ has the factorization
\[
\PS_H(q) =
\frac{(1+q^3)(1+q^6)(1+q^9)(1+q^5)(1+q^{10})(1+q^{15})
      (1-q^8)(1-q^{12})}
     {(1-q)^2}.
\]
Note that this still has positive coefficients.  There would seem to be no
particular reason why such Poincar\'e series of non-parabolic quasiparabolic
subgroups should have positive coefficients, but we do not know a
counterexample.
\end{eg}

\begin{eg}\label{eg:E}
  Inside $E_8$, apart from quasiparabolic subgroups of parabolic subgroups,
  the above index $270$ example, and the index $113400$ example coming from
  perfect involutions, there are essentially two more quasiparabolic
  subgroups (apart from those obtained via Theorem
  \ref{thm:extend_by_simples}).  The larger of the two (i.e., with the
  smaller $E_8$-set) corresponds to a subgroup of $E_8\times A_1$
  isomorphic to $\F_2^8\semidirect \AGL_3(2)$, of index $4050$.  Again, the
  highest root induces a special matching, so we obtain an action of
  $\tilde{E}_8\times A_1$.  The other corresponds to an even subgroup of
  $E_8$ abstractly isomorphic to $H_4$, of index 48384; as in the
  $F_4\subset E_6$ case above, again the Poincar\'e series of the subgroup
  is the same as its Poincar\'e series as a Coxeter group, despite the fact
  that it is far from being a reflection subgroup of $E_8$.

  For $E_7$ and $E_6$, there are no quasiparabolic subgroups other than
  those already mentioned, or those they produce via Theorem
  \ref{thm:extend_by_simples}.
\end{eg}

\begin{eg}\label{eg:geom}
  The case of fixed-point-free involutions in $A_{2n-1}$ has a nice
  geometric interpretation due to Richardson and Springer
  \cite{RichardsonRW/SpringerTA:1990}.  Let $k$ be an algebraically closed
  field of characteristic $\ne 2$, and consider the algebraic group
  $G=\GL_{2n}(k)$, with the Borel subgroup $B$ of upper triangular
  matrices.  It is classical that the double cosets $B\backslash G/B$ are in
  natural bijection with $A_{2n-1}$, with dimension essentially given by
  length.  If we now consider the subgroup $H=\Sp_{2n}(k)\subset G$, we may
  consider instead the double cosets $B\backslash G/H$.  It turns out that
  these are in natural bijection with fixed-point-free involutions in
  $A_{2n-1}$, or somewhat more naturally, with conjugates of the diagram
  automorphism.  Moreover, the dimension of such a double coset is (up to
  an additive constant) given by height.  The action of $A_{2n-1}$ is
  somewhat tricky to reconstruct from the geometry, but the action of the
  corresponding Hecke monoid is straightforward: if one decomposes
\[
B s B \iota H
\]
into double cosets (where $B\iota H$ denotes the double coset identified
with $\iota$), there will be a unique double coset of maximal dimension,
which gives the image of $\iota$ under the monoid action of $s$.  More
generally, the Bruhat order on the given conjugacy class of involutions
corresponds to inclusion of closures of double cosets.

One can also obtain the corresponding Hecke algebra module for $q$ a prime
power by considering the same double cosets, but now over the finite field
$\F_q$.  It appears that similarly the double cosets
\[
B\backslash \GL_{2n}(\F_q)/\GL_n(\F_{q^2}),
\]
where we map $\F_{q^2}\to\Mat_2(\F_q)$ in the obvious way, give rise to the
Hecke algebra module corresponding to the dual quasiparabolic set.

The (extended) affine Weyl group $\tilde{A}_{2n-1}$ has a similar geometric
interpretation in terms of double cosets of the Iwahori subgroup of
$\GL_{2n}(K)$ where $K$ is now a local field with residue field $\F_q$.
One could then replace one of the Iwahori subgroups with $\Sp_{2n}(K)$ or
$\GL_n(L)$ for either of the two quadratic extensions $L/K$, and consider
the resulting double cosets.  In each case, a back of the envelope
calculation suggests that the double cosets are classified by suitable
conjugacy classes of involutions.  This gives rise to three conjecturally
quasiparabolic actions of $\tilde{A}_{2n-1}$.  The corresponding
(conjectural) ideals in the Hecke algebra were used in \cite{vanish}, along
with two analogous ideals in $H_{\tilde{C}_{2n}}$.
\end{eg}

\bibliographystyle{plain}

\end{document}